\documentclass[11pt]{amsart}

\usepackage{amscd}
\usepackage{amsmath, amssymb}
\usepackage{amsfonts}
\newcommand{\de}{\partial}
\newcommand{\db}{\overline{\partial}}

\newcommand{\ddt}{\frac{\partial}{\partial t}}

\newcommand{\ddbar}{\sqrt{-1} \partial \overline{\partial}}
\newcommand{\Ric}{\mathrm{Ric}}
\newcommand{\ov}[1]{\overline{#1}}
\newcommand{\mn}{\sqrt{-1}}

\newcommand{\tr}[2]{\textrm{tr}_{#1}{#2}}
\newcommand{\ti}[1]{\tilde{#1}}
\newcommand{\vp}{\varphi}

\newcommand{\ve}{\varepsilon}

\renewcommand{\leq}{\leqslant}
\renewcommand{\geq}{\geqslant}
\renewcommand{\le}{\leqslant}
\renewcommand{\ge}{\geqslant}

\newcommand{\be}{\begin{equation}}
\newcommand{\ee}{\end{equation}}

\begin{document}
\newtheorem{claim}{Claim}
\newtheorem{theorem}{Theorem}[section]
\newtheorem{lemma}[theorem]{Lemma}
\newtheorem{corollary}[theorem]{Corollary}
\newtheorem{proposition}[theorem]{Proposition}
\newtheorem{question}{question}[section]
\newtheorem{conjecture}[theorem]{Conjecture}

\theoremstyle{definition}
\newtheorem{remark}[theorem]{Remark}

\title[The Chern-Ricci flow on complex surfaces]{The Chern-Ricci flow on complex surfaces}

\author[V. Tosatti]{Valentino Tosatti$^{*}$}
\thanks{$^{*}$Supported in part by NSF grants DMS-1236969, a Sloan Research Fellowship and by a Blavatnik Award for Young Scientists.}
\address{Department of Mathematics, Northwestern University, 2033 Sheridan Road, Evanston, IL 60201}
\author[B. Weinkove]{Ben Weinkove$^{\dagger}$}
\thanks{$^{\dagger}$Supported in part by NSF grant DMS-1105373.  The second-named author is currently on leave from
 the Mathematics Department of the University of California, San Diego.}
\begin{abstract}
The Chern-Ricci flow
is an evolution equation of Hermitian metrics by their Chern-Ricci form, first introduced by Gill.
Building on our previous work, we investigate this flow on complex surfaces.  We establish new estimates in the case of finite time non-collapsing, anologous to some known results for the K\"ahler-Ricci flow.  This provides evidence that the Chern-Ricci flow carries out blow-downs of exceptional curves on non-minimal surfaces.
 We also describe explicit solutions to the Chern-Ricci flow for various non-K\"ahler surfaces.  On Hopf surfaces and Inoue surfaces these solutions, appropriately normalized, collapse to a circle in the sense of Gromov-Hausdorff.  For non-K\"ahler properly elliptic surfaces, our explicit solutions collapse to a Riemann surface.  Finally, we define a Mabuchi energy functional for complex surfaces with vanishing first Bott-Chern class and show that it decreases along the Chern-Ricci flow.
 \end{abstract}

\maketitle

\section{Introduction}

The Chern-Ricci flow is a flow of Hermitian metrics on a complex manifold by their Chern-Ricci form.  It was
 introduced by Gill \cite{G} in the setting of manifolds with vanishing first Bott-Chern class.  In \cite{TW3}, the authors  proved a number of further properties of the Chern-Ricci flow, several of which are analogous to those of the K\"ahler-Ricci flow.  We continue this study here, but restrict to  complex dimension two where some additional structures can be exploited.
    Our aim is to provide more evidence that the Chern-Ricci flow is a natural evolution equation on complex manifolds and that its behavior reflects the underlying geometry.  Other flows of Hermitian metrics have been previously studied by Streets-Tian \cite{StT1, StT2, StT3}, motivated in part by the open problem of classifying Class VII surfaces.  Ultimately, our hope is that the Chern-Ricci flow may  be used as a tool in classification problems.  However, our goals in the current paper are more modest:  we wish to investigate the behavior of the flow in cases where the geometry of the manifold is already well-understood.

Let $M$ be a compact complex surface, and let $g_0$ be a Gauduchon metric on $M$.  Namely, $g_0$ is a Hermitian metric whose associated $(1,1)$ form $\omega_0 = \sqrt{-1} (g_0)_{i \ov{j}}dz_i \wedge d\ov{z}_j$ satisfies
$$\partial \ov{\partial} \omega_0=0.$$
A well-known result of Gauduchon states that every Hermitian metric on $M$ is conformal to a unique Gauduchon metric.

The \emph{Chern-Ricci flow} $\omega=\omega(t)$ starting at $\omega_0$ is a flow of Gauduchon metrics \begin{equation} \label{crf}
\ddt{\omega} = - \Ric(\omega), \qquad \omega|_{t=0} = \omega_0,
\end{equation}
where $\Ric(\omega)$ is the \emph{Chern-Ricci form} of $\omega = \sqrt{-1} g_{i\ov{j}} dz_i \wedge d\ov{z}_j$, defined by
$$\Ric(\omega) = -\ddbar \log \det g.$$
The Chern-Ricci flow makes sense for metrics which are merely Hermitian, and in any dimension (as do flows proposed in for  example \cite{StT2, LY}).  However, for the purposes of this paper, we will restrict to the case of Gauduchon metrics in complex dimension two.
Note that if $\omega_0$ satisfies the stronger condition of being $d$-closed (i.e. K\"ahler) then  (\ref{crf}) coincides with the K\"ahler-Ricci flow.

It was shown in \cite[Theorem 1.3]{TW3} that a unique maximal solution to (\ref{crf}) exists for $[0,T)$ for a number $T \in (0, \infty]$ determined by $\omega_0$.   If the volume of $M$ with respect to $\omega(t)$ tends to zero as $t \rightarrow T$, we say that the Chern-Ricci flow is \emph{collapsing} at time $T$.  Otherwise, we say that the Chern-Ricci flow is \emph{non-collapsing}.   If $T=\infty$, it is sometimes convenient to normalize the flow and consider $\omega(t)/t$ as $t\rightarrow \infty$.   If the volume of $M$ with respect to $\omega(t)/t$ tends to zero as $t \rightarrow \infty$ we say that the \emph{normalized Chern-Ricci flow} is collapsing.

The first part of this paper is concerned with non-collapsing for the flow in finite time, while in the second part of the paper we give a number of explicit examples of collapsing in both finite and infinite time.  In the last, short, section of the paper we define the Mabuchi energy functional on surfaces with vanishing first Bott-Chern class and show that it is decreasing along the Chern-Ricci flow.

\medskip

{\bf Finite time non-collapsing}. \   A natural conjecture \cite{TW3}, extending results of J. Song and the second-named author in the K\"ahler case \cite{SW1, SW2, SW3}, is that if the Chern-Ricci flow is non-collapsing in finite time, then it blows down finitely many $(-1)$ curves and continues in a unique way on a new complex surface $M$.  We require global Gromov-Hausdorff convergence of the metrics, and smooth convergence away from the $(-1)$ curves.  For more details see  Section \ref{sectionthm2}.  The goal of the first part of this paper is to make some steps towards proving this conjecture.

Suppose that the Chern-Ricci flow is non-collapsing at time $T<\infty$.  Then it was shown in \cite[Section 6]{TW3} that
there exist finitely many disjoint $(-1)$ curves $E_1, \ldots, E_k$ on $M$ giving rise to a map $\pi: M \rightarrow N$ onto a complex surface $N$ blowing down each $E_i$ to a point $y_i \in N$.  Write $M' = M \setminus \bigcup_{i=1}^k E_i$ and $N'= N \setminus \{ y_1, \ldots, y_k\}$.  Then the map $\pi$ gives an isomorphism from $M'$ to $N'$.

Our first result is as follows:

\begin{theorem} \label{thm1}  Suppose that the Chern-Ricci flow (\ref{crf}) is non-collapsing at time $T<\infty$.  Then
with the notation above, as $t \rightarrow T^-$, the metrics $g(t)$ converge to a smooth Gauduchon metric $g_T$ on $M'$ in $C^{\infty}_{\emph{loc}}(M')$.
\end{theorem}

\begin{remark}
(1) In the K\"ahler case, this result is due to Tian-Zhang \cite{TZ}.

(2)  In \cite[Theorem 1.6]{TW3}, we proved this result under additional hypotheses (equivalent to condition $(*)$ below).

(3) As described in \cite[Theorem 1.5]{TW3}, finite time non-collapsing for the Chern-Ricci flow is a common occurence.  In particular, whenever $M$ is a non-minimal complex surface with Kodaira dimension not equal to $-\infty$, there will be finite time non-collapsing for all choices of initial $\omega_0$.

(4) To prove Theorem \ref{thm1} we establish a version of the so-called \emph{Tsuji trick} \cite{Ts} in the setting of the Chern-Ricci flow and make use of some arguments of   \cite{TZ}.  \end{remark}

For the next result, we recall some notation from \cite{TW3}.  Define a family of $\partial \ov{\partial}$-closed $(1,1)$ forms
\begin{equation} \label{alphat}
\alpha_t := \omega_0 - t \textrm{Ric}(\omega_0), \quad \textrm{for } t\in [0,T].
\end{equation}  The non-collapsing condition is equivalent to the condition $\int_M \alpha_T^2>0$.   The content of our next result is that we can prove a Gromov-Hausdorff convergence result for $(M, g(t))$ as $t \rightarrow T^-$ if $\alpha_T$ is the pull-back of a $(1,1)$ form on $N$, modulo the image of $\partial \ov{\partial}$ on $M$.

\begin{theorem} \label{thm2} Suppose that the Chern-Ricci flow (\ref{crf}) is non-collapsing at time $T<\infty$.    In addition, we impose the condition:
$$(*) \quad \textrm{there exists $f \in C^{\infty}(M,\mathbb{R})$ and a smooth real $(1,1)$ form $\beta$ on $N$ with }$$
$$\alpha_T + \ddbar f = \pi^* \beta,$$
using the notation above.

Then there exists a distance function $d_T$ on $N$ such that $(N, d_T)$ is a compact metric space and $(M, g(t))$ converges in the Gromov-Hausdorff sense to $(N, d_T)$ as $t \rightarrow T^-$.  In particular, the diameter of $(M, g(t))$ is uniformly bounded from above as $t \rightarrow T^-$.
\end{theorem}

\begin{remark}
(1)  In the K\"ahler case, condition $(*)$ holds automatically, and this result is contained in the work of J. Song and the second-named author \cite{SW1}  (see the discussion at the end of Section \ref{sectionthm2}).  Our proof of Theorem \ref{thm2} makes use of several arguments from \cite{SW1}.

(2) It is not difficult to construct initial data $\omega_0$ so that $(*)$ holds.  See Remark \ref{remarki} below.

(3) Condition $(*)$ will not hold for general choices of $\omega_0$.  Indeed we will show in Proposition \ref{cruxes} that it is equivalent to $d\omega_0=\pi^*(d\beta)$ for some $\beta$, which implies that $d\omega_0\equiv 0$ on the exceptional divisors of $\pi$, and this last condition does not hold in general (see Remark \ref{remarkdw} below).

(4) On the other hand there is another condition which is weaker than $(*)$ and always holds, see Proposition \ref{weakstar}.
\end{remark}

We give the proofs of Theorems \ref{thm1} and \ref{thm2} in Sections \ref{sectionthm1} and \ref{sectionthm2} respectively.

\medskip

{\bf Examples of collapsing.}  \ In the second part of this paper we give a number of explicit examples of collapsing for the Chern-Ricci flow on non-K\"ahler surfaces.
First of all recall that as a consequence of the Kodaira-Enriques classification \cite{bhpv}, all minimal non-K\"ahler compact complex surfaces fall into the following classes:
\begin{enumerate}
\item Kodaira surfaces,
\item Minimal non-K\"ahler properly elliptic surfaces,
\item Surfaces of class VII with $b_2(M)=0,$
\item Minimal surfaces of class VII with $b_2(M)>0,$
\end{enumerate}
where a Kodaira surface is a minimal surface with $b_1(M)$ odd and Kodaira dimension $0$, a surface of class VII is a surface with $b_1(M)=1$ and Kodaira dimension $-\infty$, while a properly elliptic surface is an elliptic surface with Kodaira dimension $1$. Furthermore thanks to \cite{Bog, Ko, LYZ, T0}
we know that the surfaces in $(3)$ are all either Hopf surfaces (i.e. with universal cover $\mathbb{C}^2\backslash\{0\}$) or Inoue surfaces \cite{In}. Explicit examples of Gauduchon metrics are known for all surfaces in $(1), (2)$, for all Inoue surfaces and for some Hopf surfaces, thanks to \cite{Va, Wa, GO}. Less explicit Hermitian metrics on some surfaces in $(4)$ were constructed in \cite{FP, LeB, Br2}.

Our goal is to construct explicit solutions of the Chern-Ricci flow on surfaces in $(2), (3)$ (see Remark \ref{remm} (4) for the  class $(1)$), and to determine their
Gromov-Hausdorff limits as time approaches the maximal existence time of the flow.
We consider a family of Hopf surfaces, the Inoue surfaces, and  non-K\"ahler properly elliptic surfaces.

\begin{remark} We do not address class (4) in this paper.  These surfaces are of great interest since, except for the case $b_2(M)=1$ \cite{T1}, they are not yet completely classified. Unfortunately, it appears to be more difficult to write down explicit metrics on these manifolds and we could  not find  solutions to the Chern-Ricci flow along the lines of those we found for (2) and (3).  We plan to investigate class (4)  in future work.
\end{remark}

 First some notation.
Let $H_{\alpha, \beta}$ be the Hopf surface $H_{\alpha, \beta}=(\mathbb{C}^2\setminus \{0\})/\sim$, where
$$(z,w) \sim (\alpha z, \beta w),$$
for complex numbers $\alpha, \beta$ with $|\alpha|=|\beta| \neq 1$.     We will show that for all $H_{\alpha, \beta}$ we can find explicit solutions of the Chern-Ricci flow which collapse to a circle in the sense of Gromov-Hausdorff in finite time.  In the case of Inoue surfaces we will find examples of the normalized Chern-Ricci flow collapsing in infinite time to a circle, and for non-K\"ahler properly elliptic surfaces, collapsing to a Riemann surface.
This is interesting because collapsing to a circle never happens for the K\"ahler-Ricci flow on K\"ahler surfaces,
where collapsed limits spaces always have even real dimension. On the other hand, it is also interesting to compare this to results of Lott for the Ricci flow in real dimension $3$ \cite{Lo1, Lo2}, where the collapsed Gromov-Hausdorff limits at infinity of the normalized Ricci flow on geometric $3$-manifolds (in the sense of Thurston) are determined \cite[Theorem 1.2]{Lo2}. The complex surfaces we consider also have geometric structures (with compatible complex structures), and in fact they are precisely all the surfaces in classes (2) and (3) which have complex geometric structures \cite{Wa}, and the behavior of the Chern-Ricci flow that we discover is very similar to the behavior of the Ricci flow on geometric $3$-manifolds.

More precisely we prove:

\pagebreak[3]
\begin{theorem}\label{surfaces}  We have
\begin{enumerate}
\item[(a)] Let $H=H_{\alpha, \beta}$ be the Hopf surface as described above.  Then there exists an explicit solution $\omega(t)$ of the Chern-Ricci flow on $H$ for $t \in [0,1/2)$ with $$(H,\omega(t))\overset{GH}{\to} (S^1,d), \quad \textrm{as } t \rightarrow 1/2,$$ where $d$ is the standard distance function on the unit circle $S^1\subset\mathbb{R}^2$.
\item[(b)] Let $S$ be any Inoue surface.   Then there exists an explicit solution $\omega(t)$ of the Chern-Ricci flow on $S$ for $t \in [0,\infty)$ with $$\left( S,\frac{\omega(t)}{t} \right) \overset{GH}{\to} (S^1,d), \quad \textrm{as } t \rightarrow \infty,$$ where $d$ is the standard distance function on the unit circle $S^1\subset\mathbb{R}^2$.
\item[(c)] Let $\pi:S\to C$ be any non-K\"ahler minimal properly elliptic surface. Then there exists an explicit solution $\omega(t)$ of the Chern-Ricci flow on $S$ for $t \in [0,\infty)$ with $$\left( S,\frac{\omega(t)}{t} \right) \overset{GH}{\to} (C,d_{\mathrm{KE}}), \quad \textrm{as } t \rightarrow \infty,$$ where $d_{\mathrm{KE}}$ is the distance function on the Riemann surface $C$ induced by an orbifold K\"ahler-Einstein metric $\omega_{\mathrm{KE}}$ on $C$ which satisfies $\Ric(\omega_{\mathrm{KE}})=-\omega_{\mathrm{KE}}$ away from the images of the multiple fibers of $\pi$. We also have that $\pi^*\omega_{\mathrm{KE}}$ is a smooth form on $S$ and $\frac{\omega(t)}{t}\to \pi^*\omega_{\mathrm{KE}}$ smoothly on $S$.
\end{enumerate}
\end{theorem}

\begin{remark}\label{remm}
(1) It was shown in \cite[Theorem 1.5]{TW3} that for any initial $\omega_0$, the Chern-Ricci flow  collapses in finite time for all Hopf surfaces (e.g. $H_{\alpha, \beta}$ as above) and the normalized Chern-Ricci flow collapses in infinite time for all Inoue surfaces and properly elliptic surfaces.

(2) The example in (a) above was given already in \cite[Proposition 1.8]{TW3}.  What is new here is that we prove the Gromov-Hausdorff convergence to a circle. In fact, we will prove the same result for a family of higher-dimensional Hopf manifolds.

(3) The example in (c) should be compared with Song-Tian's results on the behavior of the K\"ahler-Ricci flow on a K\"ahler elliptic surface $\pi:S\to C$ over a curve $C$ of genus at least $2$  \cite{ST} (see also \cite{FZ, GTZ}).

(4) It is also not difficult to write down explicit, but less interesting, solutions of the Chern-Ricci flow on the Kodaira surfaces.  Indeed, there are explicit Chern-Ricci flat Gauduchon metrics on all these manifolds \cite[(1.3)]{Va} and these give trivial solutions to the flow.  In general, it was shown by Gill \cite{G} that, in any dimension, whenever the first Bott-Chern class vanishes the Chern-Ricci flow converges to a Chern-Ricci flat metric.  Gill's result makes use of the $C^0$ estimate of the authors \cite{TW2} for the Hermitian complex Monge-Amp\`ere equation (see also \cite{Ch, GL, DK, Bl}).
\end{remark}

The proof of Theorem \ref{surfaces} occupies Sections \ref{secthopf}, \ref{sectinoue1}, \ref{sectinoue2}, \ref{sectinoue3} and \ref{sectelliptic}.

\bigskip

{\bf The Mabuchi energy.} \
Finally, in Section \ref{sectionmab}, we give a definition of  the Mabuchi energy functional for complex surfaces with vanishing first Bott-Chern class.   The Mabuchi energy is a well-known object in K\"ahler geometry.
 We show that this functional decreases along the Chern-Ricci flow.  This result can be used to give an alternative proof of the convergence part of a theorem of Gill \cite{G}, in the case of Gauduchon surfaces.

\section{Proof of Theorem \ref{thm1}} \label{sectionthm1}

Suppose we are in the setting of Theorem \ref{thm1}.
Let $\omega(t)$ be the solution of the Chern-Ricci flow (\ref{crf}), which by assumption exists for $t\in [0,T)$ with $0<T<\infty$.
As stated in the introduction the non-collapsing condition implies that there is a surjective holomorphic map $\pi: M \rightarrow N$ blowing down disjoint $(-1)$ curves $E_1, \ldots, E_k$.  Indeed  the exceptional curves $E_i$ are precisely the irreducible curves on $M$ satisfying $\int_{E_i} \alpha_T=0$ for $\alpha_T = \omega_0 - T \textrm{Ric}(\omega_0)$ (see Section 6 of \cite{TW3}).
 For simplicity we assume that there is just one $(-1)$ curve $E$ which gets mapped by $\pi$ to the point $y_0 \in N$ (the general case follows from the same proof).

We first will pick some good reference metrics $\hat{\omega}_t$.  To do so, we need the following lemma, which is an essential ingredient in establishing a version of Tsuji's trick \cite{Ts} for the Chern-Ricci flow (see Lemma \ref{lem}, part (ii) and Lemmas \ref{sch} and \ref{away} below).

\begin{lemma} \label{lemmaalpha}  There exists a smooth Hermitian metric $h$ on the line bundle $[E]$ such that for any sufficiently small $\ve>0$ we can find a smooth function $f$ on $M$ such that
\begin{equation}
\alpha_T - \ve R_h + \ddbar f>0,
\end{equation}
where $R_h$ is the curvature of $h$ and $\alpha_t$ is given by (\ref{alphat}).
\end{lemma}
\begin{proof}
The non-collapsing condition together with the results of \cite[Section 6]{TW3} imply that
$\alpha_T = \omega_0 - T \textrm{Ric}(\omega_0)$ satisfies
$$\int_M \alpha_T^2 >0, \quad \int_E \alpha_T =0, \quad \int_C\alpha_T>0,$$
for all irreducible curves $C\subset M$ different from $E$.
We also have that
$$\int_M \alpha_T \wedge \omega' >0,$$
for any Gauduchon metric $\omega'$ on $M$. Indeed,
$$\int_M \alpha_T \wedge \omega'=\lim_{t\to T^-}\int_M \omega(t)\wedge \omega'\geq 0,$$
and the case $\int_M \alpha_T \wedge \omega'=0$ cannot happen since Buchdahl's ``Hodge Index Theorem"  \cite[Lemma 4]{Bu1} would imply that
$\int_M \omega'^2\leq 0$.

Let $h$ be a smooth Hermitian metric on the line bundle $[E]$.  Since $[E]$ has self-intersection $-1$, its
 curvature $R_h$ satisfies $\int_E  R_h =-1$.  If we pick $\ve>0$ small enough and we put $\alpha_{T, \ve} = \alpha_T - \ve R_h$ then
$$\int_M \alpha_{T, \ve}^2 >0, \quad \int_M \alpha_{T, \ve} \wedge \omega' >0.$$
We claim that, after possibly changing the Hermitian metric $h$ and choosing $\ve$ slightly smaller,
\begin{equation} \label{claimCalphaT}
\int_C \alpha_{T, \ve}>0, \quad \textrm{for all irreducible curves $C$ with $C^2<0$.}
\end{equation}
Given the claim, it follows from  Buchdahl's Nakai-Moishezon criterion \cite{Bu2} that there exists a smooth function $f$ (depending on $\ve$) such that $\alpha_{T, \ve} + \ddbar f >0$, as required.  We separate the proof of  the claim into two cases:

\emph{(i) $M$ is non-K\"ahler}. \ In this case $M$ has only finitely many curves of negative self-intersection (see Remark 3.3 in \cite{T2}).
Let $C$ be any such curve.   Then either $C=E$ or else $\int_C \alpha_T>0$, because we know that $E$ is the only curve whose intersection with $\alpha_T$ is zero.
We have that $\int_E \alpha_{T,\ve}=-\ve E\cdot E=\ve>0$, and if $C$ is different from $E$ then
$$\int_C \alpha_{T,\ve}=\int_C\alpha_T -\ve C\cdot E.$$
Since there are only finitely many such curves $C$, it follows that we can choose $\ve>0$ small so that
$$\int_C \alpha_{T,\ve}>0,$$
for all such $C$.
This completes the proof of the claim (\ref{claimCalphaT}) in the non-K\"ahler case.

\emph{(ii) $M$ is K\"ahler}. \ We assume now that $M$ is a K\"ahler surface, which implies also that $N$ is K\"ahler, and fix K\"ahler metrics $\omega_M, \omega_N$ on $M, N$ respectively.  We  apply Buchdahl's Corollary 9 in \cite{Bu1} to the $\partial \ov{\partial}$-closed $(1,1)$ form $\alpha_T$ which shows that there exists a $(0,1)$ form $\gamma$ and a $d$-closed $(1,1)$ form $a$ such that
$$\alpha_T + \partial \gamma + \ov{\partial \gamma} = a.$$
By the definition of the blow-down map $\pi$, we may write the deRham class $[a]$ as $[a] = [\pi^* \beta] + c[E]$ for a $d$-closed $(1,1)$ form $\beta$ on $N$ and some $c \in \mathbb{R}$ (see for example \cite[Theorem I.9.1]{bhpv}).  Intersecting with $E$ we see that $c=0$ and hence
\begin{equation}\label{alphaTbeta}
\alpha_T + \partial \gamma + \ov{\partial \gamma} = \pi^* \beta,
\end{equation}
for a possibly different form $\gamma$.

We wish to show that $[\beta]$ is a K\"ahler class on $N$, and we will use the Nakai-Moiszhezon criterion in the K\"ahler case due to Buchdahl \cite{Bu1} and Lamari \cite{La}.  From (\ref{alphaTbeta}) we infer that
\begin{equation}\label{del}
\de\alpha_T+\de\ov{\de\gamma}=0, \qquad
\db\alpha_T+\db\de\gamma=0.
\end{equation}

For any irreducible curve $C\subset N$ with $C^2<0$ we have
\begin{equation}\label{onebeta}
\int_C \beta=\int_{\pi^*C} (\alpha_T+\de\gamma+\ov{\de\gamma})=\int_{\pi^*C} \alpha_T>0,
\end{equation}
using Stokes' Theorem.

Next calculate
$$\int_N \beta^2=\int_M \pi^*\beta^2=\int_M \alpha_T^2+2\int_M\alpha_T\wedge(\de\gamma+\ov{\de\gamma})
+2\int_M \de\gamma\wedge\ov{\de\gamma}.$$
Using \eqref{del} we have
$$2\int_M \alpha_T\wedge \de\gamma=-2\int_M \de\alpha_T\wedge\gamma=
2\int_M \de\ov{\de\gamma}\wedge\gamma=-2\int_M \ov{\de\gamma}\wedge\de\gamma=
-2\int_M \de\gamma\wedge\ov{\de\gamma},$$
and
$$2\int_M \alpha_T\wedge \ov{\de \gamma}=-2\int_M \db\alpha_T\wedge\ov{\gamma}=
2\int_M \db\de\gamma\wedge\ov{\gamma}=
-2\int_M \de\gamma\wedge\ov{\de\gamma},$$
and so
$$\int_N \beta^2=\int_M \alpha_T^2-2\int_M \de\gamma\wedge\ov{\de\gamma}.$$
But because $\int_M\omega_M^2>0$ and $\int_M \omega_M \wedge (\de\gamma+\ov{\de\gamma})=0$
we can apply  \cite[Lemma 4]{Bu1} and conclude that
$$\int_M (\de\gamma+\ov{\de\gamma})^2=2\int_M  \de\gamma\wedge\ov{\de\gamma}\leq 0,$$
and so
\begin{equation}\label{threebeta}
\int_N \beta^2\geq \int_M \alpha_T^2>0.
\end{equation}
Next,
$$\int_N \beta\wedge\omega_N=\int_M \pi^*\beta\wedge\pi^*\omega_N=
\int_M \alpha_T\wedge\pi^*\omega_N = \lim_{t\to T^-}\int_M \omega(t)\wedge \pi^*\omega_N\geq 0,$$
since $d\omega_N=0$.  For $\delta>0$ sufficiently small $\omega_N + \delta \beta$ is K\"ahler and
\begin{equation} \label{twobeta}
\int_N \beta \wedge (\omega_N + \delta \beta) \ge \delta \int_N \beta^2 >0.
\end{equation}
Therefore combining \eqref{onebeta},  \eqref{threebeta} and \eqref{twobeta} we can
apply the Nakai-Moishezon criterion of Buchdahl and Lamari to conclude that there exists a function $f$ on $N$
such that $\ti{\beta}=\beta+\ddbar f$ is K\"ahler on $N$. It follows from the construction of the blow-down map (see for example \cite[p.187]{GH}) that we may pick a Hermitian metric $h$ on $[E]$ such that
$\pi^*\ti{\beta}-\ve R_h$ is K\"ahler on $M$ for all $\ve>0$ small. Then with this choice of $h$, for any irreducible curve $C\subset M$
with $C^2<0$, we have
$$\int_C \alpha_{T, \ve} = \int_C (\alpha_T -\ve R_h)=\int_C (\pi^*\ti{\beta}-\ve R_h)>0,$$
because $\int_C \de\gamma=0$ by Stokes' Theorem.  This finishes the proof of the claim (\ref{claimCalphaT}) and the lemma.
\end{proof}

Define $\hat{\omega}_T = \alpha_T + \ddbar f$ with $f$ given by the lemma.  Note that in general $\hat{\omega}_T$ is not a metric, but by the lemma $\hat{\omega}_T - \ve R_h$ is a metric.  Define reference forms $\hat{\omega}_t = \frac{1}{T} ( (T-t)\omega_0 + t \hat{\omega}_T)$.  By shrinking $\ve>0$ if necessary we may assume without loss of generality that $\omega_0 - \ve R_h > \frac{1}{2} \omega_0$.   Hence
\begin{equation} \label{ohatlb}
\hat{\omega}_t - \ve R_h = \frac{1}{T} ((T-t) (\omega_0 - \ve R_h) + t (\hat{\omega}_T - \ve R_h)) \ge c_0 \omega_0 >0,
\end{equation}
for some $c_0>0$.  Note that from now on we assume that $\ve>0$ is fixed.

This argument above crucially uses the fact that we are in complex dimension 2.  However, since the calculations that follow do not require this restriction on dimension, we write $n$ instead of $2$.

If we let $\varphi$ solve
$$\ddt{} \varphi = \log \frac{\omega(t)^n}{\Omega}, \quad \varphi|_{t=0} =0,$$
with $\Omega = \omega_0^n e^{f/T}$ then we can write the solution of the Chern-Ricci flow $\omega(t)$ as  $\omega(t) = \hat{\omega}_t + \ddbar \varphi$ (see Section 4 of \cite{TW3}).  Fix a holomorphic section $s$ of $[E]$ vanishing to order 1 along $E$, so that on $M'=M\backslash E$ we have that $R_h=-\mn\de\db\log |s|^2_h.$ Given what we have proved, the following result can now be established using the arguments of \cite{TZ} for the K\"ahler-Ricci flow:

\begin{lemma}\label{lem}
There exists a uniform $C$ such that
\begin{enumerate}
\item[(i)] $\varphi \le C$.
\item[(ii)] Let $\tilde{\varphi} = \varphi - \ve \log|s|^2_h$.  Then $\tilde{\varphi} \ge -C$.
\item[(iii)] $\dot{\varphi} \le C$.
\end{enumerate}
\end{lemma}
\begin{proof}
For (i),  note that $\hat{\omega}_t$ is bounded from above for $t \in [0,T]$.  Then the upper bound of $\varphi$ follows from the maximum principle applied to the equation
$$\ddt{} \varphi  = \log \frac{(\hat{\omega}_t + \ddbar \varphi)^n}{\Omega},$$
as in Lemma 4.1 of \cite{TW3}.

For (ii), first note that $\tilde{\varphi} \rightarrow \infty$ along $E$ and hence for each fixed time $t$, the function $x \mapsto \tilde{\varphi}(x,t)$ attains a minimum at some point in $M'$.  Then compute at the minimum of $\tilde{\varphi}$,
\begin{eqnarray*}
\ddt{} \tilde{\varphi} & = & \log \frac{(\hat{\omega}_t - \ve R_h + \ddbar \tilde{\varphi})^n}{\Omega} \\
& \ge & \log \frac{\left( c_0 \omega_0 \right)^n}{\Omega} \ge  - C,
\end{eqnarray*}
where we have used the estimate (\ref{ohatlb}).  The lower bound of $\ti{\varphi}$ follows from the minimum principle.

Part (iii) follows from considering the evolution of $Q= t\dot{\varphi} - \varphi-nt$ as used in the K\"ahler-Ricci flow in \cite{TZ} (in the case of the Chern-Ricci flow, see the second part of Lemma 4.1 in \cite{TW3}).  Indeed,
$$\left( \ddt{} - \Delta \right) Q = - \tr{\omega}{\omega_0} \le 0,$$
so that by the maximum principle, $Q$ is uniformly bounded from above.  Using (i),  $\dot{\varphi}$ is bounded from above.
\end{proof}

One comment about notation. If $g$ and $g'$ are Hermitian metrics with corresponding $(1,1)$ forms
$\omega$ and $\omega'$, then we will write interchangeably
$$\tr{g}{g'}=\tr{\omega}{\omega'},$$
for the trace of the metric $g'$ with respect to $g$. Next:

\begin{lemma}\label{sch}
There exist uniform positive constants $C, A$ such that
$$\omega^n \ge \frac{1}{C} |s|^{2A\ve}_h \omega_0^n.$$
\end{lemma}
\begin{proof}
We apply the maximum principle to
$$Q = \log \frac{\omega_0^n}{\omega^n} - A\tilde{\varphi},$$
for $A$ a constant to be determined.  Note that $Q \rightarrow -\infty$ on $E$.
Compute at a point of $M \setminus E$,
\begin{eqnarray*}
\left( \ddt{} - \Delta \right) Q & = & \tr{\omega}{\Ric(\omega_0)} - A \dot{\varphi}  + A \tr{\omega}{(\omega - (\hat{\omega}_t - \ve R_h))}\\
& = & - \tr{\omega}((A-1) (\hat{\omega}_t -\ve R_h) - \textrm{Ric}(\omega_0)) - A \log\frac{\omega^n}{\Omega} \\
&& - \tr{\omega}{(\hat{\omega}_t -\ve R_h)} + An.
\end{eqnarray*}
From (\ref{ohatlb}) we may
choose $A$ sufficiently large so that for all $t \in [0,T]$,
$$(A-1) (\hat{\omega}_t - \ve R_h) -  \textrm{Ric}(\omega_0)  \ge  \omega_0.$$
Note that by the arithmetic-geometric means inequality,
$$\tr{\omega}{(\hat{\omega}_t - \ve R_h)} \ge c_0 \tr{\omega}{\omega_0} \ge c \left( \frac{\Omega}{\omega^n} \right)^{1/n},$$
for a uniform $c>0$.  Then
\begin{eqnarray*}
\left( \ddt{} - \Delta \right) Q & \le & - \tr{\omega}{\omega_0} + A \log \frac{\Omega}{\omega^n} - c \left( \frac{\Omega}{\omega^n} \right)^{1/n}  + An \\
& \le & - \tr{\omega}{\omega_0} + C,
\end{eqnarray*}
using the fact that $x \mapsto A \log x- cx^{1/n}$ is bounded from above for $x>0$.  It follows that $\tr{\omega}{\omega_0}\le C$ at the maximum of $Q$ and hence $\omega_0^n/\omega^n$ is uniformly bounded from above at this point.  But note that $-\tilde{\varphi}$ is uniformly bounded from above and hence $Q$ is uniformly bounded from above, and the result follows.
\end{proof}

In the next lemma we make use of a trick of Phong-Sturm \cite{PS2}, which we  employed in our previous paper \cite{TW3}.

\begin{lemma}\label{away}
There exist uniform positive constants $C, A$ such that
$$\mathrm{tr}_{g_0} \, g \le \frac{C}{|s|_h^{2A\ve}}.$$
\end{lemma}
\begin{proof} Choose a constant $C_0$ so that $\tilde{\varphi} +C_0 \ge 1$. We compute the evolution of
$$Q = \log \tr{g_0}{g} - A \tilde{\varphi} + \frac{1}{\tilde{\varphi}+C_0},$$
for $A$ to be determined (assume at least that $A \ve >1)$.  The idea of Phong-Sturm \cite{PS2}, used in their study of the complex Monge-Amp\`ere equation, is to make use of the quantity $1/(\tilde{\varphi}+C_0)$.
 Note that $1/(\tilde{\varphi}+C_0)$  is bounded between 0 and 1.

From Lemma \ref{lem}(i) it is sufficient to show that $Q$ is bounded from above.  Observe that $Q$ tends to negative infinity on $E$.
From Proposition 3.1 of \cite{TW3} (see also equation (4.2) of \cite{TW3}) we have
\begin{align} \label{eqnfrombigcalc}
\left( \ddt{} - \Delta \right) \log \tr{g_0}{g} & \le  \frac{2}{(\tr{g_0}{g})^2} \textrm{Re} \left( g^{\ov{\ell}k}  (T_0)^p_{kp}  \partial_{\ov{\ell}} \tr{g_0}{g} \right) + C \tr{g}{g_0},
\end{align}
assuming, without loss of generality, that we are calculating at a point with $\tr{g_0}{g} \ge 1$.
To bound the first term on the right hand side, we note that at a maximum point of $Q$ we have $\partial_i Q=0$ and hence
\begin{equation} \nonumber
\frac{1}{\tr{g_0}{g}} \partial_i \tr{g_0}{g} - A \partial_i \tilde{\varphi} - \frac{1}{(\tilde{\varphi}+C_0)^2} \partial_i \tilde{\varphi}=0.
\end{equation}
Thus at this maximum point for $Q$,
\begin{align} \nonumber \lefteqn{
\left| \frac{2}{(\tr{g_0}{g})^2} \textrm{Re} \left( g^{\ov{\ell}k}  (T_0)^p_{kp}  \partial_{\ov{\ell}} \tr{g_0}{g} \right) \right| } \\ \nonumber
& \le \left| \frac{2}{\tr{g_0}{g}} \textrm{Re} \left( \left(A+ \frac{1}{(\tilde{\varphi}+C_0)^2}\right) g^{\ov{\ell} k} (T_0)^p_{kp} (\partial_{\ov{\ell}} \tilde{\varphi} )  \right) \right|  \\ \nonumber
& \le \frac{ | \partial \tilde{\varphi}|^2_{g}}{(\tilde{\varphi} +C_0)^3} + C A^2 (\tilde{\varphi}+C_0)^3 \frac{\tr{g}{g_0}}{(\tr{g_0}{g})^2},
\end{align}
for a uniform constant $C$.
If at the maximum of $Q$ we have $(\tr{g_0}{g})^2\leq A^2(\tilde{\varphi}+C_0)^3$ then at
the same point we have
$$Q\leq \log A +\frac{3}{2}\log(\ti{\vp}+C_0)-A\ti{\vp}+\frac{1}{\tilde{\varphi}+C_0}\leq C_A,$$
for a constant $C_A$ depending on $A$, and we are done.  If on the other hand at the maximum of $Q$ we have $A^2(\tilde{\varphi}+C_0)^3 \le (\tr{g_0}{g})^2$ then
\begin{align} \nonumber
\left| \frac{2}{(\tr{g_0}{g})^2} \textrm{Re} \left( g^{\ov{\ell}k}  (T_0)^p_{kp}  \de_{\ov{\ell}} \tr{g_0}{g} \right) \right|
& \le \frac{ | \partial \tilde{\varphi}|^2_{g}}{(\tilde{\varphi} +C_0)^3}  + C \tr{g}{g_0}.
\end{align}
Now compute at the maximum of $Q$, using (\ref{eqnfrombigcalc}) and part (iii) of Lemma \ref{lem},
\begin{align} \nonumber
0 & \le \left( \ddt{} - \Delta \right) Q \\ \nonumber
& \le \frac{ | \partial \tilde{\varphi}|^2_{g}}{(\tilde{\varphi} +C_0)^3}  +C \tr{g}{g_0} - \left(A +\frac{1}{(\tilde{\varphi} +C_0)^2}\right) \dot{\varphi}  \\ \nonumber
& \ \ \  + \left( A + \frac{1}{(\tilde{\varphi}+C_0)^2} \right)\tr{\omega}{(\omega - (\hat{\omega}_t - \ve R_h))}\\ \nonumber
&\ \ \  - \frac{2}{ (\tilde{\varphi}+C_0)^3} | \partial \tilde{\varphi}|^2_g\\
& \leq C \tr{g}{g_0} +(A+1)\log \frac{\Omega}{\omega^n}+(A+1)n -A\tr{\omega}{(\hat{\omega}_t - \ve R_h)} +C.
\end{align}
But recall that we have that $\hat{\omega}_t - \ve R_h \ge c_0 \omega_0$, and so we may
choose $A$ sufficiently large so that at that point
$$\tr{g}{g_0} \le C \log \frac{\Omega}{\omega^n} +C.$$
Hence at the maximum of $Q$,
$$\tr{g_0}{g} \le \frac{1}{(n-1)!} (\tr{g}{g_0})^{n-1} \frac{\det g}{\det g_0} \le
C\frac{\omega^n}{\Omega} \left(\log \frac{\Omega}{\omega^n}\right)^{n-1}+C\leq C',$$
because we know that $\frac{\omega^n}{\Omega}\leq C$ (Lemma \ref{lem} (iii)) and $x\mapsto x|\log x|^{n-1}$
is bounded above for $x$ close to zero. From part (ii) of Lemma \ref{lem}, this implies that $Q$ is bounded from above at its maximum, hence everywhere.  This completes the proof of the lemma.
\end{proof}

Combining Lemmas \ref{sch} and \ref{away} gives uniform bounds above and below away from zero for $\omega(t)$ on compact subsets of $M'$.  To obtain $C^{\infty}_{\textrm{loc}}(M')$ estimates, we apply the local estimates of Gill \cite[Section 4]{G}. Convergence follows immediately from this, as in the proof of Theorem 1.6 of \cite{TW3}.
 This completes the proof of Theorem \ref{thm1}.

\section{Proof of Theorem \ref{thm2}} \label{sectionthm2}

Assume the hypotheses of Theorem \ref{thm2}.
As discussed above, we know that
there exist finitely many disjoint $(-1)$-curves $E_1,\dots,E_k$ such that $\int_{E_i}\alpha_T=0$, where we recall that $\alpha_t$ is given by (\ref{alphat}). As in the previous section we assume for simplicity that $k=1$ and write $E$ for the $(-1)$ curve. By assumption $(*)$, there exists a function
$f$ on $M$ and a smooth real $(1,1)$ form $\beta$ on $N$ such that
\begin{equation}\label{crux}
\alpha_T+\ddbar f = \pi^*\beta.
\end{equation}

\begin{remark} \label{remarki} It is straightforward to construct initial data $\omega_0$ on $M$ so that $(*)$ holds.  Indeed, let $\omega_M$ and $\omega_N$ be any Gauduchon metrics on $M$ and $N$ respectively, and fix $T>0$.  Then we claim that for $C>0$ sufficiently large, there exists a smooth function $\tilde{f}$ on $M$ so that
$$\omega_0 := C \pi^*\omega_N  + T \Ric (\omega_M) + \ddbar \tilde{f}$$
is Gauduchon, and the flow starting at $\omega_0$ will be non-collapsing at $T$, satisfying $(*)$ with $\beta = C \omega_N$.  Indeed, as the reader can verify, it is enough to check  that $\omega_0$ is positive definite.
If $\pi: M \rightarrow N$ is the blow up map, then the canonical bundles on $M$ and $N$ are related by $K_M = \pi^*K_N+[E]$.  It follows that we can define a smooth Hermitian metric $h$ on $[E]$ by $h   \omega_M^n= \pi^* \omega_N^n$.  Then
$$C \pi^* \omega_N + T \Ric (\omega_M) = C \pi^* \omega_N + T \pi^* \Ric(\omega_N) - T R_h,$$
and for $C>0$ sufficiently large, we have $\frac{1}{2}C  \pi^* \omega_N + T \pi^* \Ric(\omega_N) \ge 0$ and (see for example \cite{GH}, p.187) $\frac{1}{2}C \pi^*\omega_N - T R_h + \ddbar \tilde{f}>0$ for some smooth function $\tilde{f}$ on $M$.  Thus with these choices of $C$ and $\tilde{f}$, $\omega_0$ is positive definite.  \end{remark}

We first show that, after replacing $f$ by another smooth function, we may assume that $\beta$ is a Gauduchon metric on $N$.

\begin{lemma}  There exists a smooth function $f'$ and a Gauduchon metric $\omega_N$ on $N$ such that
\begin{equation}\label{crux2}
\alpha_T+\ddbar f' = \pi^*\omega_N.
\end{equation}
\end{lemma}
\begin{proof}
From (\ref{crux}) it immediately follows that $\de\db\beta=0$. Furthermore, we obviously have $\int_N \beta^2>0$ and $\int_C \beta=\int_{\pi^* C}\alpha_T>0$ for all curves $C$ in $N$.
Let $\gamma$ be any Gauduchon metric on $N$. Then
$$\int_N \beta\wedge\gamma=\int_M \pi^*\beta\wedge\pi^*\gamma=\int_M \alpha_T\wedge\pi^*\gamma=
\lim_{t\to T^-}\int_M\omega(t)\wedge\pi^*\gamma\geq 0.$$
For $\delta>0$ sufficiently small $\gamma+\delta\beta$ is positive definite and
$$\int_N \beta\wedge(\gamma+\delta\beta)\geq \delta\int_N\beta^2>0.$$
Therefore Buchdahl's Nakai-Moishezon criterion \cite{Bu2} gives us a function $h$ on $N$ such that $\beta+\ddbar h>0$. Defining $f'=f-\pi^*h$ and $\omega_N = \beta + \ddbar h$ we then obtain (\ref{crux2}).
\end{proof}

As a consequence of this, we have that $\pi^*\omega_N-\omega_0$ is a $d$-closed form, so that $d\omega_0=\pi^*(d\omega_N)$.
This implies that the torsion tensors of $\omega_0$ and $\pi^*\omega_N$ are related by
\begin{equation}\label{useful}
(T_0)^p_{j\ell}(g_0)_{p\ov{k}}=(\de\omega_0)_{j\ov{k}\ell}=(\pi^*\de\omega_N)_{j\ov{k}\ell}=(\pi^*T_N)^p_{j\ell}(\pi^*g_N)_{p\ov{k}}.
\end{equation}
This equality is crucial in the arguments that follow.
We may choose a
 smooth metric $h$ on the fibers of $[E]$ with curvature $R_h$, and $\ve_0>0$ small such that
\begin{equation} \label{oNR}
\pi^*\omega_N-\ve_0 R_h >0,
\end{equation}
Indeed this follows again from the argument in for example p. 187 of \cite{GH}.

As in the previous section, we write $s$ for a defining section of $E$, and
since the calculations that follow do not require the dimension to be $2$, we write $n$ instead of $2$.
 From Lemma 2.4 of \cite{SW1} we  have that
\begin{equation} \label{simpleSW}
\omega_0\leq \frac{C}{|s|^2_h}\pi^*\omega_N.
\end{equation}
Furthermore if we define $\hat{\omega}_t=\frac{1}{T}((T-t)\omega_0+t\pi^*\omega_N)$ then from (\ref{oNR}), we have
\begin{equation} \label{lbhot}
\hat{\omega}_t-\ve_0 R_h\geq c_0\omega_0, \quad  \textrm{for all} \quad 0\leq t\leq T,
\end{equation}
for a uniform $c_0>0$.
Moreover $\omega(t)=\hat{\omega}_t+\ddbar\vp(t)$ for $t<T$ where $\vp$ solves
$$\frac{\de}{\de t}\vp=\log\frac{\omega(t)^n}{\Omega},\quad \varphi|_{t=0} =0,$$
with $\Omega=\omega_0^n e^{f/T}$. Since $\hat{\omega}_T = \pi^* \omega_N \ge 0$, we can apply Proposition 5.1 of \cite{TW3} to obtain:

\begin{lemma} \label{lemmaphiagain}
There exists a uniform constant $C$ such that
\begin{enumerate}
\item[(i)] $\displaystyle{| \varphi |\le C}$.
\item[(ii)] $\displaystyle{\dot{\varphi} \le C.}$
\end{enumerate}
\end{lemma}

Note that we already have (ii) from Lemma \ref{lem}.  The point of Lemma \ref{lemmaphiagain} is that the condition $\hat{\omega}_T\ge 0$ gives us a lower bound for $\varphi$.
The following two lemmas give analogs of the estimates of  Lemma 2.5 (i) of \cite{SW1}.  There are significant additional technical difficulties arising from the torsion terms, and we need to employ again the Phong-Sturm trick \cite{PS2}.   First:

\begin{lemma}\label{bds}
There exists a uniform constant $C>0$ such that
$$\omega\leq\frac{C}{|s|^2_h}\pi^*\omega_N.$$
\end{lemma}
\begin{proof}
For notational simplicity denote by $\hat{\omega}=\pi^*\omega_N$, which is a metric on $M'=M \setminus E$. First, we apply the calculation of  Proposition 3.1 of \cite{TW3} to obtain, on $M'$,
\[\begin{split}
\left( \ddt{} - \Delta \right) \log \tr{\hat{g}}{g} &\le \frac{2}{(\tr{\hat{g}}{g})^2} \textrm{Re} \left(
\hat{g}^{\ov{\ell}i}g^{\ov{q}k} (T_0)^p_{ki}(g_0)_{p\ov{\ell}} \de_{\ov{q}} \mathrm{tr}_{\hat{g}}g\right) +C\tr{g}{\hat{g}}\\
&- \frac{1}{\tr{\hat{g}}g} \bigg[g^{\ov{j}i}\hat{g}^{\ov{\ell}k}\left(\hat{\nabla}_i  \left(  \ov{(T_0)^p_{j\ell}} (g_0)_{k \ov{p}} \right) +
 \hat{\nabla}_{\ov{\ell}} \left( (T_0)^p_{ik} (g_0)_{p \ov{j}} \right) \right)\\
&-g^{\ov{j}i}\hat{g}^{\ov{\ell}k}\ov{\hat{T}^q_{j\ell}} (T_0)^p_{ik} (g_0)_{p \ov{q}}\bigg],
\end{split}\]
where $C$ depends only on the curvature and torsion of $\omega_N$, and where $\hat{\nabla}$ denotes the pullback of the Chern connection of $\omega_N$.
We now use the identity \eqref{useful} four times in the last equation to get
\begin{align} \nonumber
\left( \ddt{} - \Delta \right) \log \tr{\hat{g}}{g} &\le \frac{2}{(\tr{\hat{g}}{g})^2} \textrm{Re} \left(
g^{\ov{q}k} \hat{T}^i_{ki} \de_{\ov{q}} \mathrm{tr}_{\hat{g}}g\right) +C\tr{g}{\hat{g}} \\ \nonumber
&- \frac{1}{\tr{\hat{g}}g} \bigg[g^{\ov{j}i}\hat{\nabla}_i  \ov{\hat{T}^\ell_{j\ell}} +
 g^{\ov{j}i} \hat{g}^{\ov{\ell}k}\hat{g}_{p \ov{j}} \hat{\nabla}_{\ov{\ell}} \hat{T}^p_{ik}-g^{\ov{j}i}\hat{g}^{\ov{\ell}k} \hat{g}_{p \ov{q}}\ov{\hat{T}^q_{j\ell}} \hat{T}^p_{ik}\bigg]\\ \label{log2}
&\leq \frac{2}{(\tr{\hat{g}}{g})^2} \textrm{Re} \left(
g^{\ov{q}k} \hat{T}^i_{ki} \de_{\ov{q}} \mathrm{tr}_{\hat{g}}g\right) +C\tr{g}{\hat{g}}+\frac{C\tr{g}{\hat{g}}}{\tr{\hat{g}}g}.
\end{align}
Take now $\delta>0$ and consider
$$Q_\delta= \log \tr{\hat{g}}{g} +\log|s|^{2(1+\delta)}_h -A\vp+\frac{1}{\ti{\vp}+C_0},$$
where $\ti{\vp}=\vp-\frac{1+\delta}{A}\log|s|^2_h$ and $\ti{\vp}+C_0\geq 1$,
so $Q_\delta$ goes to negative infinity as $x$ tends to $E$. Our goal is to prove that $Q_\delta\leq C$ independent of
$\delta$, since we can then let $\delta$ go to zero and we are done. It is obvious that at the maximum of $Q$ we can assume that $\tr{\hat{g}}g\geq 1$.

At this point the proof proceeds exactly in the same way as in Lemma \ref{away}.
Indeed, at a maximum point of $Q_\delta$ we have
\begin{equation} \nonumber
\frac{1}{\tr{\hat{g}}{g}} \partial_i \tr{\hat{g}}{g} - A \partial_i \tilde{\varphi} - \frac{1}{(\tilde{\varphi}+C_0)^2} \partial_i \tilde{\varphi}=0,
\end{equation}
for all $i$. Thus at this point,
\begin{align} \nonumber \lefteqn{
\left| \frac{2}{(\tr{\hat{g}}{g})^2} \textrm{Re} \left(
g^{\ov{q}k} \hat{T}^i_{ki} \de_{\ov{q}} \mathrm{tr}_{\hat{g}}g\right)  \right| } \\ \nonumber
& \le \left| \frac{2}{\tr{\hat{g}}{g}} \textrm{Re} \left( \left(A+ \frac{1}{(\tilde{\varphi}+C_0)^2}\right) g^{\ov{q} k} \hat{T}^i_{ki} (\partial_{\ov{q}} \tilde{\varphi} )  \right) \right|  \\ \nonumber
& \le \frac{ | \partial \tilde{\varphi}|^2_{g}}{(\tilde{\varphi} +C_0)^3} + C A^2 (\tilde{\varphi}+C_0)^3 \frac{\tr{g}{\hat{g}}}{(\tr{\hat{g}}{g})^2},
\end{align}
for a uniform constant $C$.
If at the maximum of $Q_\delta$ we have $(\tr{\hat{g}}{g})^2\leq A^2(\tilde{\varphi}+C_0)^3$ then at
the same point we have
$$Q_\delta\leq \log A +\frac{3}{2}\log(\ti{\vp}+C_0)-A\ti{\vp}+\frac{1}{\tilde{\varphi}+C_0}\leq C_A,$$
for a constant $C_A$ depending on $A$, and we are done.  If on the other hand at the maximum of $Q_\delta$ we have $A^2(\tilde{\varphi}+C_0)^3 \le (\tr{\hat{g}}{g})^2$ then
\begin{align} \nonumber
\left| \frac{2}{(\tr{\hat{g}}{g})^2} \textrm{Re} \left(
g^{\ov{q}k} \hat{T}^i_{ki} \de_{\ov{q}} \mathrm{tr}_{\hat{g}}g\right) \right|
& \le \frac{ | \partial \tilde{\varphi}|^2_{g}}{(\tilde{\varphi} +C_0)^3}  + C \tr{g}{\hat{g}}.
\end{align}
Now compute at the maximum of $Q_\delta$, using \eqref{log2} and part (ii) of Lemma \ref{lemmaphiagain},
\begin{align} \nonumber
0 & \le \left( \ddt{} - \Delta \right) Q_\delta \\ \nonumber
& \le \frac{ | \partial \tilde{\varphi}|^2_{g}}{(\tilde{\varphi} +C_0)^3}  +C \tr{g}{\hat{g}} - \left(A +\frac{1}{(\tilde{\varphi} +C_0)^2}\right) \dot{\varphi}  \\ \nonumber
& \ \ \  + \left( A + \frac{1}{(\tilde{\varphi}+C_0)^2} \right)\tr{\omega}{\left(\omega - \hat{\omega}_t + \frac{(1+\delta)R_h}{A} \right)} - \frac{2}{ (\tilde{\varphi}+C_0)^3} | \partial \tilde{\varphi}|^2_g\\ \nonumber
& \leq C \tr{g}{\hat{g}} +(A+1)\log \frac{\Omega}{\omega^n}+(A+1)n -A\tr{\omega}{\left(\hat{\omega}_t - \frac{(1+\delta)R_h}{A} \right)} +C
\end{align}
For $A$ sufficiently large, we have from (\ref{lbhot}),
$$\hat{\omega}_t - \frac{(1+\delta) R_h}{A} \ge c_0 \omega_0.$$
Hence, recalling that $\tr{g}{\hat{g}}\leq C\tr{g}{g_0},$ we see that we may choose $A$ sufficiently large (and independent of $\delta$) so that at that point
$$\tr{g}{g_0} \le C \log \frac{\Omega}{\omega^n} +C.$$
Hence at the maximum of $Q_\delta$,
$$\tr{g_0}{g} \le \frac{1}{(n-1)!} (\tr{g}{g_0})^{n-1} \frac{\det g}{\det g_0} \le
C\frac{\omega^n}{\Omega} \left(\log \frac{\Omega}{\omega^n}\right)^{n-1}+C\leq C',$$
because we know that $\frac{\omega^n}{\Omega}\leq C$ (Lemma \ref{lemmaphiagain} (ii)) and $x\mapsto x|\log x|^{n-1}$
is bounded above for $x$ close to zero.
On the other hand from \eqref{simpleSW} and part (i) of Lemma \ref{lemmaphiagain}, this implies that $Q_\delta$ is bounded from above at its maximum, hence everywhere, uniformly in $\delta$.
Using Lemma \ref{lemmaphiagain} (i) again completes the proof of the lemma.
\end{proof}

We will make use of Lemma \ref{bds} to prove the following:

\begin{lemma} \label{lemmaradial}
There exists a uniform $\eta>0$ and $C>0$ such that
\begin{equation} \label{radialestimate}
\omega \le \frac{C}{|s|^{2(1-\eta)}_h} \omega_0.
\end{equation}
\end{lemma}
\begin{proof}
As in the previous lemma, we write $\hat{\omega} = \pi^* \omega_N$.  Define on $M'$,
\begin{equation}
Q_{\delta} = \log \tr{g_0}{g} + A \log \left((\tr{\hat{g}}{g}) |s|^{2(1+\delta)}_h\right) - A^2 \varphi + \frac{1}{\tilde{\psi}+\tilde{C}} + \frac{1}{\tilde{\varphi}+\tilde{C}},
\end{equation}
where $\tilde{\psi}$ and $\tilde{\varphi}$ are defined by
$$\tilde{\psi} := - \log \left((\tr{\hat{g}}{g}) |s|^{2(1+\delta)}_h\right) + A \varphi, \quad \tilde{\varphi} := \varphi - \frac{1+\delta}{A} \log |s|^{2}_h$$
so that in particular
\begin{equation} \label{psi}
\tilde{\psi}=  - \log (\tr{\hat{g}}{g}) + A \tilde{\varphi},
\end{equation}
and we may write
\begin{equation} \label{Qalt}
Q_{\delta} = \log \tr{g_0}{g} - A\tilde{\psi} + \frac{1}{\tilde{\psi}+\tilde{C}} + \frac{1}{\tilde{\varphi}+\tilde{C}}.
\end{equation}
The constant $\tilde{C}$ is chosen so that $\tilde{\psi}+\tilde{C} \ge 1$ and $\tilde{\varphi}+\tilde{C} \ge 1$ (we can find such a constant $\tilde{C}$ because $|\varphi|$ is bounded, and $(\tr{\hat{g}}{g}) |s|_h^2$ is bounded from above by Lemma \ref{bds}).  The constant $A>0$ is to be determined.  Observe that $Q_{\delta}$ is  the quantity used in Lemma 2.5 of \cite{SW1} with the addition of two ``Phong-Sturm terms'' $1/(\tilde{\psi}+\tilde{C})$ and $1/(\tilde{\varphi}+\tilde{C})$.  We have that $Q_{\delta}$ tends to negative infinity along $E$.

We wish to show that  $Q_{\delta}$ is bounded from above.  From Lemmas \ref{lemmaphiagain} and  \ref{bds}, it suffices to show that
 we have a uniform upper bound for $\tr{g_0}{g}$, independent of $\delta$ at a point where $Q_{\delta}$ achieves a maximum.   Recall from (\ref{eqnfrombigcalc}) and (\ref{log2}), there exists a uniform constant $C_0$ such that
\begin{equation} \label{log1again}
\left( \ddt{} - \Delta \right) \log \tr{g_0}{g} \le C_0 \tr{g}{g_0} + \frac{2}{(\tr{g_0}{g})^2} \textrm{Re} \left( g^{\ov{q}k} (T_0)^i_{ki} \partial_{\ov{q}} \tr{g_0}{g} \right),
\end{equation}
and
\begin{equation} \label{log2again}
\left( \ddt{} - \Delta \right) \log \tr{\hat{g}}{g} \le C_0 \tr{g}{\hat{g}} + \frac{2}{(\tr{\hat{g}}{g})^2} \textrm{Re} \left( g^{\ov{q}k} \hat{T}^i_{ki} \partial_{\ov{q}} \tr{\hat{g}}{g} \right),
\end{equation}
where we are assuming, without loss of generality, that we are working at a point with $\tr{g_0}{g} \ge 1$ and $\tr{\hat{g}}{g} \ge 1$ (note that  $\tr{g_0}{g} \le C \tr{\hat{g}}{g}$).  Observe that the inequality (\ref{log2again}) makes use of  condition $(*)$.

Compute using (\ref{Qalt}), (\ref{log1again}),
\begin{align*}
\left( \ddt{} - \Delta\right)Q_{\delta} & \le C_0 \tr{g}{g_0} +  \frac{2}{(\tr{g_0}{g})^2} \textrm{Re} \left( g^{\ov{q}k} (T_0)^i_{ki} \partial_{\ov{q}} \tr{g_0}{g} \right) \\
& -  \left( A + \frac{1}{(\tilde{\psi}+\tilde{C})^2} \right) \left( \ddt{} - \Delta \right) \tilde{\psi} \\
& - \frac{1}{(\tilde{\varphi}+\tilde{C})^2} \left( \ddt{} - \Delta \right) \tilde{\varphi} \\
& - \frac{2}{(\tilde{\psi}+\tilde{C})^3} | \partial \tilde{\psi}|^2_g - \frac{2}{(\tilde{\varphi}+\tilde{C})^3} | \partial \tilde{\varphi}|^2_g.
\end{align*}
Hence, using (\ref{log2again}),
\begin{align*}
\left( \ddt{} - \Delta\right)Q_{\delta}
& \le  C_0 \tr{g}{g_0} +  \frac{2}{(\tr{g_0}{g})^2} \textrm{Re} \left( g^{\ov{q}k} (T_0)^i_{ki} \partial_{\ov{q}} \tr{g_0}{g} \right)  + C_0(A +1) \tr{g}{\hat{g}} \\
& +  \left( A + \frac{1}{(\tilde{\psi}+\tilde{C})^2} \right) \cdot \frac{2}{(\tr{\hat{g}}{g})^2} \textrm{Re} \left( g^{\ov{q}k} \hat{T}^i_{ki} \partial_{\ov{q}} \tr{\hat{g}}{g} \right)  \\
& - \left( A \left(A + \frac{1}{(\tilde{\psi}+\tilde{C})^2} \right) + \frac{1}{(\tilde{\varphi}+\tilde{C})^2} \right) \dot{\varphi} \\
& + \left( A \left(A + \frac{1}{(\tilde{\psi}+\tilde{C})^2} \right) + \frac{1}{(\tilde{\varphi}+\tilde{C})^2} \right) \tr{\omega}{\left( \omega - \hat{\omega}_t + \frac{(1+\delta)R_h}{A} \right)} \\
& - \frac{2}{(\tilde{\psi}+\tilde{C})^3} | \partial \tilde{\psi}|^2_g - \frac{2}{(\tilde{\varphi}+\tilde{C})^3} | \partial \tilde{\varphi}|^2_g.
\end{align*}
For all $A$ sufficiently large, we have from (\ref{lbhot}),
$$\frac{1}{2} A \hat{\omega}_t - (1+\delta) R_h \ge c_0 \omega_0.$$
Hence, we may choose $A$ sufficiently large (and independent of $\delta$) so that
\begin{align*}
A^2 \hat{\omega}_t - A(1+\delta) R_h &= \frac{1}{2} A^2 \hat{\omega}_t + A\left(  \frac{1}{2} A \hat{\omega}_t - (1+\delta) R_h \right) \\
& \ge C_0(A+1) \hat{\omega} + (C_0+1) \omega_0,
\end{align*}
for $C_0$ as above.  Hence for this choice of $A$,
\begin{align*}
\left( \ddt{} - \Delta\right)Q_{\delta} & \le    - \tr{g}{g_0} +  \frac{2}{(\tr{g_0}{g})^2} \textrm{Re} \left( g^{\ov{q}k} (T_0)^i_{ki} \partial_{\ov{q}} \tr{g_0}{g} \right)   \\
& +  \left( A + \frac{1}{(\tilde{\psi}+\tilde{C})^2} \right) \cdot \frac{2}{(\tr{\hat{g}}{g})^2} \textrm{Re} \left( g^{\ov{q}k} \hat{T}^i_{ki} \partial_{\ov{q}} \tr{\hat{g}}{g} \right)  \\
& - B \dot{\varphi} - \frac{2}{(\tilde{\psi}+\tilde{C})^3} | \partial \tilde{\psi}|^2_g - \frac{2}{(\tilde{\varphi}+\tilde{C})^3} | \partial \tilde{\varphi}|^2_g +C',
\end{align*}
for $B= A^2+A+1$, where we are making use of Lemma \ref{lemmaphiagain} (ii).
At a point where $Q_{\delta}$ achieves its maximum, we have
$$\frac{1}{\tr{g_0}{g}} \partial_{\ov{q}} \tr{g_0}{g} = \left(A + \frac{1}{(\tilde{\psi}+\tilde{C})^2}\right) \partial_{\ov{q}} \tilde{\psi} + \frac{1}{(\tilde{\varphi}+\tilde{C})^2} \partial_{\ov{q}} \tilde{\varphi}.
$$
Observe also that from (\ref{psi})
$$\frac{1}{\tr{\hat{g}}{g}} \partial_{\ov{q}} \tr{\hat{g}}{g} = - \partial_{\ov{q}} \tilde{\psi} + A \partial_{\ov{q}} \tilde{\varphi}.$$
Hence
\begin{align*}
\left( \ddt{} - \Delta\right)Q_{\delta} & \le    - \tr{g}{g_0} +  \frac{C''(A+1)^2(\tilde{\psi}+\tilde{C})^3}{(\tr{g_0}{g})^2} \tr{g}{g_0} + \frac{1}{(\tilde{\psi}+\tilde{C})^3} | \partial \tilde{\psi}|^2_g \\
& + \frac{C''}{(\tr{g_0}{g})^2} \tr{g}{g_0} + \frac{1}{(\tilde{\varphi}+\tilde{C})^3} | \partial \tilde{\varphi}|^2_g
  \\
 & + \frac{C''(A+1)^2(\tilde{\psi}+\tilde{C})^3}{(\tr{\hat{g}}{g})^2} \tr{g}{\hat{g}} +  \frac{1}{(\tilde{\psi}+\tilde{C})^3} | \partial \tilde{\psi}|^2_g \\
& + \frac{C''(A+1)^4(\tilde{\varphi}+\tilde{C})^3}{(\tr{\hat{g}}{g})^2} \tr{g}{\hat{g}} +  \frac{1}{(\tilde{\varphi}+\tilde{C})^3} | \partial \tilde{\varphi}|^2_g \\
& - B \dot{\varphi} - \frac{2}{(\tilde{\psi}+\tilde{C})^3} | \partial \tilde{\psi}|^2_g - \frac{2}{(\tilde{\varphi}+\tilde{C})^3} | \partial \tilde{\varphi}|^2_g +C'.
\end{align*}
But we may assume without loss of generality that at the maximum of $Q_{\delta}$, we have $(\tr{g_0}{g})^2 \ge 4C'' (A+1)^2 (\tilde{\psi}+\tilde{C})^3$ otherwise $Q_{\delta}$ is bounded from above and we're done.  Similarly, since $\tr{g_0}{g} \le C \tr{\hat{g}}{g}$, we may assume that $(\tr{\hat{g}}{g})^2 \ge 4 C'' (A+1)^2 (\tilde{\psi}+\tilde{C})^3$ and $(\tr{\hat{g}}{g})^2 \ge 4C'' (A+1)^4 (\tilde{\varphi}+\tilde{C})^3$.
Hence
\begin{align*}
\left( \ddt{} - \Delta\right)Q_{\delta} & \le - \frac{1}{4} \tr{g}{g_0}  - B \log \frac{\omega^n}{\omega_0^n} + C',
\end{align*}
and we can conclude as in the proof of Lemma \ref{away} that $\tr{g_0}{g}$ is bounded from above at the maximum of $Q_{\delta}$.  It follows, using Lemma \ref{bds} and the uniform bound on $\varphi$  that $Q_{\delta}$ is bounded from above uniformly in $\delta$.  Letting $\delta \rightarrow 0$, we obtain
$$\log \tr{g_0}{g} + A \log (\tr{\hat{g}}{g}) |s|_h^2 \le C.$$
The lemma then follows from the same argument as in Lemma 2.5 of \cite{SW1}.  Indeed, since $\tr{g_0}{g} \le C \tr{\hat{g}}{g}$, we have
$$\log (\tr{g_0}{g})^{A+1} |s|_h^{2A} \le C,$$
and the estimate (\ref{radialestimate}) follows with $\eta = 1/(A+1)>0$.
\end{proof}

We can now easily finish the proof of Theorem \ref{thm2} following the arguments of \cite{SW1}.

\begin{proof}[Proof of Theorem \ref{thm2}] Identify a small neighborhood of $y_0 \in N$ with a small ball $B$ centered at the origin in $\mathbb{C}^2$, and consider $\omega(t)$ as a metric on $B \setminus \{ 0 \}$, then we have the following estimates:
\begin{enumerate}
\item[(i)] Let $S_r$ be a small sphere of radius $r>0$ centered at the origin in $B$.  Then the diameter of $S_r$ with respect to the metric induced from $\omega(t)$ is uniformly bounded from above, independent of $r$.
\item[(ii)]  For any $x \in B \setminus \{0 \}$, the length of a radial path $\gamma(\lambda)=\lambda x$ for $\lambda \in (0,1]$ with respect to the metric $\omega(t)$ is uniformly bounded from above by $C|x|^{\eta}$ for a uniform constant $C$.
\end{enumerate}
Indeed (i) follows from Lemma \ref{bds} and the argument of Lemma 2.7 (i) in \cite{SW1}.  For part (ii), note that from Lemma \ref{lemmaradial} we have
$$|V|^2_{\omega(t)} \le  C |s|_h^{-2(1-\eta)} |V|^2_{\omega_0} \le C' |s|_h^{2\eta},$$
for $V$ the vector field  $V=z_i \frac{\partial}{\partial z_i}$ on $B$.  This is because $|V|^2_{\omega_0} \le C|s|_h^{2}$, as can be seen by writing down an explicit metric on $B$ which is uniformly equivalent to $\omega_0$ (see e.g. \cite[Lemma 2.6]{SW1}).  Then (ii) follows from the argument of Lemma 2.7 (ii) in \cite{SW1}.

Given (i) and (ii), the proof of Theorem \ref{thm2} follows exactly as in Section 3 of \cite{SW1}.
 \end{proof}

Let us now discuss condition $(*)$ in more detail.
As before let $\pi:M\to N$ be the blowup of finitely many points $y_i\in N$ with exceptional divisors $E_i$.

\begin{proposition}\label{cruxes}
The following are equivalent:
\begin{enumerate}
\item[(*)] There exist a function
$f\in C^\infty(M)$ and a smooth real $(1,1)$ form $\beta$ on $N$ such that
\begin{equation}\label{crux3}
\alpha_T+\ddbar f = \pi^*\beta.
\end{equation}
\item[(**)] There exists a smooth real $(1,1)$ form $\beta$ on $N$ such that
\begin{equation}\label{crux4}
d\omega_0 = \pi^*(d\beta).
\end{equation}
\end{enumerate}
Furthermore, either of these implies that for any $i$ and for every point $x\in E_i$ we have that
\begin{equation}\label{zero}
(d\omega_0)_x \equiv 0.
\end{equation}
\end{proposition}

\begin{proof}
It is obvious that $(*)$ implies $(**)$. As before, for simplicity we can assume that
$\pi$ is the blowup of $N$ at one single point $y_0$, with exceptional divisor $E$.
To see that $(**)$ implies $\eqref{zero}$, fix $x\in E$ and
write $T_xM=T_xE\oplus H$ for some complementary real $2$-dimensional subspace $H$. For any vector $X\in T_xM$
while $X=X'+X''$ with $X'\in T_xE$ and $X''\in H$. Note that by definition we have $\pi_* X'=0$.
Then for any three vectors $X,Y,Z\in T_xM$ we have
\[\begin{split}
(d\omega_0)_x(X,Y,Z)&=(d\beta)_{\pi(x)}(\pi_*X,\pi_*Y,\pi_*Z)\\
&=(d\beta)_{\pi(x)}(\pi_*X'',\pi_*Y'',\pi_*Z'')=0,
\end{split}\]
since $\pi_*X'',\pi_*Y'',\pi_*Z''$ belong to a real $2$-dimensional plane, while $d\beta$ is a $3$ form.

Now we show that $(**)$ implies $(*)$. Consider the metrics along the Chern-Ricci flow
$\omega(t)=\omega_0-t\Ric(\omega_0)+\ddbar\vp(t)$. Viewing $\omega(t)$ as currents, their mass is bounded by
$$\int_M \omega(t)\wedge\omega_0=\int_M (\omega_0-t\Ric(\omega_0))\wedge\omega_0\leq C$$
independent of $t$, so by weak compactness there is a sequence
$t_i\to T$ and an $L^1$ function $\vp_T$ such that $\vp(t_i)$ converges to $\vp_T$ in $L^1$.
Furthermore, we know from Theorem \ref{thm1} that $\vp_T$ is smooth away from $E$. Call
$$\omega_T=\omega_0-T\Ric(\omega_0)+\ddbar\vp_T,$$
which is a $\de\db$-closed positive real $(1,1)$ current on $M$, smooth away from $E$. Furthermore
we have that $d\omega_T=d\omega_0$ is a smooth form.
Consider the pushforward current $\pi_* \omega_T$. It is a $\de\db$-closed positive real $(1,1)$ current on $N$, smooth away from $y_0$, and it satisfies
$$d\pi_* \omega_T=\pi_*d\omega_0=\pi_*\pi^*(d\beta).$$
Now, for any smooth $2$ form $\gamma$ on $N$, if we first pull it back  as a form $\pi^*\gamma$
and then push it forward as a current $\pi_*\pi^*\gamma$ , we get back the same current $\gamma$.
Indeed, for any test $2$ form $\psi$ on $N$ we have
\begin{equation}\label{bir}
\int_N \psi\wedge \pi_*\pi^*\gamma=\int_M \pi^*\psi\wedge\pi^*\gamma=\int_N\psi\wedge\gamma,
\end{equation}
where the last equality holds because $\pi$ is bimeromorphic.
It follows that $d\pi_*\omega_T=d\beta$, or in other words the $(1,1)$ current $\pi_*\omega_T-\beta$ is $d$-closed.
Therefore it is locally $\de\db$-exact, i.e. there exists an  $L^1$ function $F$ defined on a neighborhood
$V$ of $y_0$ such that
$$\pi_*\omega_T=\beta+\ddbar F,$$
holds as currents on $V$. Furthermore, $F$ is smooth away from $y_0$ (by regularity of the complex Laplacian $\Delta_0=\tr{\omega_0}{(\ddbar)}$). Fix a smooth cutoff function $\rho$ which is identically zero outside $V$
and identically $1$ on a smaller neighborhood $U\subset V$ of $y_0$.
Then the function $\rho F$ is defined on the whole of $N$ and smooth away from $y_0$.
Now consider $\rho F\circ \pi$, which is an  $L^1$ function on $M$, smooth away from $E$, so we can define
$\pi^*\beta+\ddbar (\rho F\circ\pi)$, which is a real $(1,1)$ current on $M$, positive on $\pi^{-1}(U)$. We have that
$$\gamma=\omega_T-\pi^*\beta-\ddbar (\rho F\circ\pi)$$ is a real $(1,1)$ current on $M$ which is
$d$-closed and smooth away from $E$. Furthermore its restriction $\gamma|_{\pi^{-1}(U)}$ is supported on $E$, and it is written as the difference of two positive currents.
This last condition implies that its coefficients are measures, and so $\gamma|_{\pi^{-1}(U)}$ is a flat current (in the terminology of \cite{Fe}), and Federer's support theorem \cite[4.1.15]{Fe} implies that $\gamma|_{\pi^{-1}(U)}=\lambda[E]$ for some
real constant $\lambda$. But integrating $\gamma$ over $E$ we see that $\lambda=0$ and so $\gamma|_{\pi^{-1}(U)}=0$.
Therefore we have that $\gamma=\pi^*\eta$ for a smooth $d$-closed real $(1,1)$ form $\eta$ on $N$.
Therefore,
\[\begin{split}
\ddbar(\rho F\circ\pi - \vp_T)&=\omega_T-\pi^*(\beta+\eta)-\ddbar\vp_T\\
&=\omega_0-T\Ric(\omega_0)-\pi^*(\beta+\eta),
\end{split}\]
which is smooth, so by regularity of $\Delta_0$, we have that $\rho F\circ\pi - \vp_T=-f$ a smooth function on $M$,
which satisfies \eqref{crux3} with $\beta$ replaced by $\beta+\eta$.
\end{proof}

\begin{remark} \label{remarkdw}  It is easy to see that the condition $d\omega_0|_E \equiv 0$ does not hold for all choices of $\omega_0$.  Indeed, fix a point $x$ in $E$ and suppose that $(d\omega_0)_x=0$.     Define a $(0,1)$ form on $\mathbb{C}^2$ by $\gamma = z_1 \overline{z_2} d\overline{z_1}$ and, by identifying a neighborhood of $x$ in $E$ with a ball centered at the origin of $\mathbb{C}^2$, and extending $\gamma$ in an arbitrary way outside of a neighborhood of $x$, we may consider $\gamma$ as a $(0,1)$ form on $M$.   Consider $\tilde{\omega}_0 = \omega_0 + \ve (\partial \gamma + \overline{\partial \gamma})$ for $\ve>0$.  As long as $\ve$ is sufficiently small, $\tilde{\omega}_0$ is a Gauduchon metric.  But one can check that $(d\tilde{\omega}_0)_x \neq 0$.
\end{remark}

Let us now remark that a weaker version of condition $(*)$ always holds.

\begin{proposition}\label{weakstar} There exist a smooth
$(0,1)$ form $\gamma$ on $M$ and a smooth real $(1,1)$ form $\beta$ on $N$ such that
\begin{equation}\label{goal}
\alpha_T+\de\gamma+\ov{\de\gamma}=\pi^*\beta.
\end{equation}
\end{proposition}

\begin{proof}  We can again assume that there is only one exceptional divisor $E$.
Recall that the Bott-Chern cohomology group of a compact complex manifold $M$ is
$$H^{1,1}_{\mathrm{BC}}(M, \mathbb{R})=\frac{\{d\textrm{-closed real }(1,1)\textrm{ forms}\}}{\{\ddbar \psi, \psi\in C^\infty(M,\mathbb{R})\}},$$
while the Aeppli cohomology group is
$$H^{1,1}_{\mathrm{A}}(M, \mathbb{R})=\frac{\{\de\db\textrm{-closed real }(1,1)\textrm{ forms}\}}{\{\de\gamma+\ov{\de\gamma}, \gamma\in \Lambda^{0,1}(M)\}}.$$
These are finite dimensional real vector spaces, and when $n=2$ they are isomorphic. In fact, $H^{1,1}_{\mathrm{BC}}(M, \mathbb{R})\cong H^{1,1}_{\mathrm{A}}(M, \mathbb{R})^*$ through the pairing $H^{1,1}_{\mathrm{BC}}(M, \mathbb{R})\times H^{1,1}_{\mathrm{A}}(M, \mathbb{R})\to\mathbb{R}$ given by wedge and integration (see e.g. \cite{Sc}).

First of all recall a few facts from \cite[pp.737-738]{Fu}: the Bott-Chern cohomology group $H^{1,1}_{\mathrm{BC}}(M, \mathbb{R})$ is isomorphic to the
group one obtains by replacing in its definition smooth forms with currents. Therefore one has not only a pullback map $\pi^*:H^{1,1}_{\mathrm{BC}}(N, \mathbb{R})\to H^{1,1}_{\mathrm{BC}}(M, \mathbb{R})$ but also a pushforward map $\pi_*:H^{1,1}_{\mathrm{BC}}(M, \mathbb{R})\to H^{1,1}_{\mathrm{BC}}(N, \mathbb{R})$ induced by the pushforward of currents. Whenever $\pi$ is bimeromorphic (such as in our case), then $\pi_* \pi^*=\mathrm{Id}$ (cfr. \eqref{bir}), which implies in particular that $\pi^*:H^{1,1}_{\mathrm{BC}}(N, \mathbb{R})\to H^{1,1}_{\mathrm{BC}}(M, \mathbb{R})$ is injective.

The exact same statements hold for the Aeppli cohomology (see e.g. \cite[p.16]{Sc}.

Now \cite[Proposition 1.1]{Fu} gives us the following exact sequence
$$0\to \mathbb{R}[E]\to H^{1,1}_{\mathrm{BC}}(M, \mathbb{R})\overset{\pi_*}{\to} H^{1,1}_{\mathrm{BC}}(N, \mathbb{R})\to 0,$$
which splits using the map $\pi^*$, and so
$$H^{1,1}_{\mathrm{BC}}(M, \mathbb{R})\cong H^{1,1}_{\mathrm{BC}}(N, \mathbb{R})\oplus \mathbb{R}.$$
Therefore we also have
$$H^{1,1}_{\mathrm{A}}(M, \mathbb{R})\cong H^{1,1}_{\mathrm{A}}(N, \mathbb{R})\oplus \mathbb{R}.$$
We wish to identify the image $\pi^*H^{1,1}_{\mathrm{A}}(N, \mathbb{R})\subset H^{1,1}_{\mathrm{A}}(M, \mathbb{R}).$
We have just proved that $\dim_{\mathbb{R}}\mathrm{coker} \pi^*=1$. We have the Poincar\'e-Lelong formula
\begin{equation}\label{pl2}
2\pi[E]=\eta_E+\ddbar\log|s|^2_h,
\end{equation}
where $\eta_E$ is a $d$-closed smooth real $(1,1)$ form cohomologous to $2\pi c_1([E])$, and so it defines a
Bott-Chern cohomology class $[\eta_E]_{\mathrm{BC}}$ with the property that for any Aeppli class $[\psi]_{\mathrm{A}}$ on $M$ we have $$\int_M \eta_E\wedge\psi=2\pi\int_E \psi.$$
Therefore we can define a linear functional
$$F:H^{1,1}_{\mathrm{A}}(M,\mathbb{R})\to\mathbb{R},\quad F([\psi]_{\mathrm{A}})=\int_M \eta_E\wedge\psi.$$
which is obviously surjective, so $\mathrm{codim}_{\mathbb{R}}\ker F=1$.
But we also have that $\pi^*H^{1,1}_{\mathrm{A}}(N, \mathbb{R})\subset \ker F$ and since we have just proved
that these two spaces have the same dimension, it follows that $\pi^*H^{1,1}_{\mathrm{A}}(N, \mathbb{R})= \ker F$.
In other words, the sequence
$$0\to H^{1,1}_{\mathrm{A}}(N, \mathbb{R})\overset{\pi^*}{\to} H^{1,1}_{\mathrm{A}}(M, \mathbb{R})
\overset{F}{\to}\mathbb{R}\to 0,$$
is exact.

In our case we have $F([\alpha_T]_{\mathrm{A}})=2\pi\int_E \alpha_T=0$, and so $[\alpha_T]_{\mathrm{A}}=\pi^*[\beta]_{\mathrm{A}}$ for
some $\de\db$-closed smooth real $(1,1)$ form $\beta$ on $N$.
By definition of Aeppli class, this is precisely \eqref{goal}.
\end{proof}

 We end this section by describing more precisely the conjecture mentioned in the introduction.
 Assume we are in the setup of Theorem \ref{thm1}, and continuing with the same notation, we expect that the following results hold:
\begin{enumerate}
\item[(i)]  As $t \rightarrow T^-$, the metrics $g(t)$ converge to a smooth Gauduchon metric $g_T$ on $M'$ in $C^{\infty}_{\textrm{loc}}(M')$ (i.e. the result of Theorem \ref{thm1}).  Using $\pi$, we may regard $g_T$ as a Gauduchon metric on $N'$.
\item[(ii)]  Let $d_{g_T}$ be the distance function on $N'$ given by $g_T$.  Then there exists a unique metric $d_T$ on $N$ extending $d_{g_T}$ such that $(N, d_T)$ is a compact metric space homeomorphic to $N$ and $(N, d_T)$ is the metric completion of $(N', d_{g_T})$
\item[(iii)]  $(M, g(t)) \rightarrow (N, d_T)$ as $t \rightarrow T^-$ in the Gromov-Hausdorff sense.
\item[(iv)]  There exists a smooth maximal solution $g(t)$ of the Chern-Ricci flow on $N$ for $t \in (T, T_N)$ with $T<T_N \le \infty$ such that $g(t)$ converges to $g_T$ as $t\rightarrow T^+$ in $C^{\infty}_{\textrm{loc}}(N')$.  Furthermore, $g(t)$ is uniquely determined by $g_0$.
\item[(v)]  $(N, g(t)) \rightarrow (N, d_T)$ as $t \rightarrow T^+$ in the Gromov-Hausdorff sense.
\end{enumerate}

These results were proved for the case of the K\"ahler-Ricci flow by Song and the second-named author \cite{SW1,SW2}.
With the terminology of \cite{SW1, SW2}, we say that $g(t)$ performs a  \emph{canonical surgical contraction} if this occurs.

If the condition $(*)$ is imposed, Theorem \ref{thm2} shows that we obtain (ii) and (iii), except for the statement about identifying $(N, d_T)$ as the metric completion of $(N', d_{g_T})$.   We do not expect that there is any fundamental obstacle to establishing (iv) and (v) under the condition $(*)$, since the methods used (estimates obtained via the maximum principle, the weak solution constructed by Song-Tian \cite{ST3}, the results of Ko\l odziej \cite{Kol}) can most likely be generalized to this setting (see e.g. \cite{DK}).  Nevertheless, there are considerable technical challenges here.

We expect there are more difficulties in proving the full statement of (iii) (including the identification of the metric completion of $(N', d_{g_T})$ as in \cite{SW2}) or removing the assumption (*).  These problems may require  new techniques.

\section{The Hopf surfaces}\label{secthopf}

In this section we give a proof of the first part of Theorem \ref{surfaces}.  In fact, we consider more general Hopf manifolds, of complex dimension $n$.  Define
$H=(\mathbb{C}^n\setminus \{0\})/\sim$, where
$$(z_1,\dots, z_n) \sim (\alpha_1 z_1, \ldots, \alpha_n z_n),$$
where $|\alpha_1|=\dots=|\alpha_n| \neq 1$.
Following \cite[Section 8]{TW3} we consider the metric
$$\omega_H = \frac{\delta_{ij}}{r^2} \mn dz_i \wedge d\ov{z}_j,$$
where $r^2= |z_1|^2+\dots+|z_n|^2$.
We know from \cite{TW3} that the metric
$$\omega(t) = \omega_H - t \Ric(\omega_H)=
\frac{1}{r^2}\left((1-nt)\delta_{ij}+nt\frac{\ov{z}_i z_j}{r^2}\right)\mn dz_i \wedge d\ov{z}_j$$
gives an explicit solution of the Chern-Ricci flow on $H$ defined on $[0,\frac{1}{n}),$ and as $t$ approaches $\frac{1}{n}$ the metrics $\omega(t)$ converge smoothly to the nonnegative $(1,1)$-form
$$\omega\left(\frac{1}{n}\right)=\frac{\ov{z}_i z_j}{r^4}\mn dz_i \wedge d\ov{z}_j.$$

\begin{theorem}As $t$ approaches $\frac{1}{n}$ we have that $(H,\omega(t))\overset{\emph{GH}}{\to} (S^1,d)$,  where $d$ is the distance function on the circle $S^1\subset\mathbb{R}^2$ with radius $\frac{\log |\alpha_1|}{\sqrt{2}\pi}$.
\end{theorem}

This proves part (a) of Theorem \ref{surfaces}, since we can always scale $\omega_0$ by a constant to obtain
as limit the unit circle.

 We recall here that Gromov-Hausdorff convergence of metrics spaces can be defined as follows \cite{Ro}.  The Gromov-Hausdorff distance $d_{\textrm{GH}}((X, d_X),(Y, d_Y))$ between two metrics spaces $(X,d_X)$ and $(Y, d_Y)$ is the infimum of all $\ve>0$  such that there exist $F: X \rightarrow Y$ and $G:Y \rightarrow X$ with
$$| d_X(x_1, x_2) - d_Y(F(x_1), F(x_2))| \le \ve \quad \forall \, x_1, x_2 \in X,$$
and $$d_X(x, G(F(x))) \le \ve \quad \forall \, x\in X,$$
together with the two symmetric properties for $Y$.  Note that $F, G$ are not required to be continuous maps.  If $d_t$ are metrics on $X$, we say that $(X,d_t) \overset{\textrm{GH}}{\to} (Y,d_Y)$ as $t \rightarrow T$ if $d_{\textrm{GH}}( (X, d_t), (Y, d_Y)) \rightarrow 0$ as $t \rightarrow T$.

If furthermore we have finite groups $H$ and $K$ acting isometrically on $(X, d_X)$ and $(Y, d_Y)$ respectively,
we can define the equivariant Gromov-Hausdorff distance between these spaces as the infimum of all $\ve>0$ such that there exist maps $F,G$ as above and maps $\phi:H\to K$ and $\psi:K\to H$ such that in addition we have
$$d_X(F(h\cdot x), \phi(h)\cdot F(x))\leq \ve, \quad d_X(F(\psi(k)\cdot x), k\cdot F(x))\leq \ve,$$
for all $x\in X, h\in H, k\in K$, together with the two symmetric properties for $Y$ (see \cite{Fuk}, \cite[Definition 1.5.2]{Ro}).

\begin{proof}  On $\mathbb{C}^n\setminus \{0 \}$ write $z_i = x_i + \sqrt{-1} y_i$.
A short calculation shows that the $(1,1)$ form $\omega(1/n)$ defines a nonnegative symmetric tensor $h$ on $\mathbb{C}^n \setminus \{0 \}$ by
$$h(X, Y) = \frac{2}{r^4} \left( (Z \cdot X)(Z \cdot Y) + (JZ \cdot X)(JZ \cdot Y)\right),$$
for $$Z= \sum_i (x_i \partial_{x_i} + y_i \partial_{y_i}), \quad JZ= \sum_i (x_i \partial_{y_i} - y_i \partial_{x_i}),$$
where $J$ is the standard complex structure on $\mathbb{C}^n$.
Define a real distribution $\mathcal{D}$ on  $\mathbb{C}^n \setminus \{0\}$ by
$$\mathcal{D} = \{ X  \ | \ h(X, Y) =0 \quad \forall \, Y\}.$$
Note that the condition $X \in \mathcal{D}$ is equivalent to $Z \cdot X=0=JZ\cdot X$.  In particular, since $Z$ is the radial vector field on $\mathbb{C}^n \setminus \{ 0 \}$, the distribution $\mathcal{D}$ is tangent to every sphere $S^{2n-1}$ centered at the origin.

Fix a sphere $S^{2n-1} \subset \mathbb{C}^n\setminus \{0 \}$.  Denote by $\mathcal{D}_{S}$ the distribution $\mathcal{D}$ restricted to this sphere.  We claim that  $\mathcal{D}_{S}$ is bracket generating, namely that $\mathcal{D}_{S}$ together with its iterated Lie brackets generate the tangent space to $TS^{2n-1}$.  To see this,  note that $\mathcal{D}_S$ is given by vectors $X \in TS^{2n-1}$ with $JZ \cdot X =0$ and hence  we may write
$$\mathcal{D}_S = \{ X \in TS^{2n-1} \ | \ \alpha(X)=0 \},$$
for $\alpha = \sum y_j dx_j - x_j dy_j$.  But this is the well-known standard contact structure on $S^{2n-1}$.  Indeed the reader can verify that $\alpha \wedge (d\alpha)^{n-1}$ is a nowhere vanishing volume form on $S^{2n-1}$.  A theorem of Carath\'eodory (see for example \cite{Mo}) shows then that $\mathcal{D}_S$ is bracket generating.  Note also that since $\mathcal{D}_S$ is orthogonal to $JZ \in TS^{2n-1}$, it follows that $TS^{2n-1}$ is spanned by the distribution $\mathcal{D}_S$ together with the vector field $JZ$.

Calculate $g_t(Z,Z) = 2$ and $$g_t(X, Y) = 2(1-nt) (X \cdot Y) /r^2, \quad \textrm{for } X \in \mathcal{D}.$$
In particular, $g_t(X, Z)=0$ for $X \in \mathcal{D}$.  It follows that $Z$ is $g_t$-orthogonal to $TS^{2n-1}$ since $g_t(Z,JZ)=0$ by the Hermitian property of $g_t$, and $\mathcal{D}_S$, $JZ$ span the tangent space to $S^{2n-1}$.

Consider the map $F:H\to S^1$ which maps the equivalence class of $(z_1, \ldots, z_n)$ in $H$ to the equivalence class of $r=\sqrt{\sum_i|z_i|^2}$ in $\mathbb{R}^+ /(r\sim |\alpha_1|r)\cong S^1$, and also the map $G:S^1\to H\cong S^1\times S^{2n-1}$ which maps a point $x$ to $(x,y)$ for some fixed element $y\in S^{2n-1}$ (identified with the unit sphere in $\mathbb{C}^n$). Note that the diffeomorphism $H\cong S^1\times S^{2n-1}$ can be realized explicitly by sending a point $\mathbf{z}=(z_1, \ldots, z_n)$ to $\left(r,\frac{\mathbf{z}}{r}\right)$. We clearly have that $F\circ G=\mathrm{Id}.$

\begin{remark}
The metric $\omega_H$ coincides (up to a universal constant factor) with the pull-back of the standard product metric on $S^1 \times S^{2n-1}$ via the isomorphism $H \cong S^1 \times S^{2n-1}$ described above.
\end{remark}

On the circle $S^1$ put the metric $2(d\log r)^2=2\frac{(dr)^2}{r^2}$, and denote by $d$ its distance function.
It is isometric to the standard metric on $S^1\subset \mathbb{R}^2$ with total length $\sqrt{2}\int_1^{|\alpha_1|}\frac{dr}{r}=\sqrt{2}\log|\alpha_1|$, and therefore radius $\frac{\log |\alpha_1|}{\sqrt{2}\pi}$.
Now the kernel of $F_*:TH\to TS^1$ is $TS^{2n-1}$, and the $g_t$-orthogonal complement of $\ker F_*$ is spanned by the radial vector field $Z= r \frac{\partial}{\partial r}$. But we also have that $$g_t(Z,Z)=2=
g_t\left(r\frac{\de}{\de r},r\frac{\de}{\de r}\right)= F^*(2(d\log r)^2)\left(r\frac{\de}{\de r},r\frac{\de}{\de r}\right).$$
Hence  $F:(H,g_t)\to (S^1,2(d\log r)^2)$ is a Riemannian submersion, i.e.  $F_*$ is an isometry when restricted to the $g_t$-orthogonal complement of $\ker F_*$. Since every Riemannian submersion is distance-nonincreasing,
\begin{equation}\label{one}
d(F(x),F(y))\leq d_t(x,y),
\end{equation}
for all $x,y\in H$ and all $0\leq t<\frac{1}{n}$, where $d_t$ denotes the distance function on $H$ induced by $g_t$.
This also shows that
\begin{equation}\label{two}
d(p,q)\leq d_t(G(p),G(q)),
\end{equation}
for all $p,q\in S^1$ and all $0\leq t<\frac{1}{n}$, because $F\circ G=\mathrm{Id}.$
Furthermore, it is clear that
\begin{equation}\label{three}
d_t(G(p),G(q))\leq d(p,q),
\end{equation}
since we can connect the points $G(p)$ and $G(q)$ by the radial path whose $g_t$-length equals $d(p,q)$.

Finally, given two points $x,y\in H$, choose representatives $x,y\in\mathbb{C}^n\setminus\{0\}$ with $1\leq |x|,|y|\leq |\alpha_1|$, and call $S^{2n-1}_\rho$ the sphere with center the origin and radius $\rho=|y|$. Call $z$ the radial projection of $x$ onto $S^{2n-1}_\rho$.
Since $\mathcal{D}_S$ is bracket-generating, Carath\'eodory's theorem \cite{Mo} implies that
 we can join the points $y$ and $z$ in $S^{2n-1}_{\rho}$ by a path in $S^{2n-1}_\rho$ with tangent vectors in $\mathcal{D}_S$. Furthermore, the length of this path with respect to the Euclidean metric $\delta_{ij}$ can be bounded by a constant $C$ independent of $x,y$ (this is because the sub-Riemannian distance induced on $S^{2n-1}$ by $\mathcal{D}$ and $\delta_{ij}$ has finite diameter since $S^{2n-1}_\rho$ is compact, see Theorem 2.3 in \cite{Mo}). But the restriction of $g_t$ on $\mathcal{D}$ equals $2\frac{(1-nt)}{|y|^2}\delta_{ij}$ and so the $g_t$-length of this path is bounded above by $C(1-nt)$.
We then join $x$ to $z$ by the radial path whose $g_t$-length equals $d(F(x),F(y))$, by the previous discussion.
Altogether we get
\begin{equation}\label{four}
d_t(x,y)\leq d_t(x,z)+d_t(z,y)\leq d(F(x),F(y))+C(1-nt).
\end{equation}

We obviously have
\begin{equation}\label{five}
d(p,F(G(p)))=0,
\end{equation}
for every $p\in S^1$. Finally pick a point $x\in H$ and consider the point $G(F(x))$. They lie in the same
fiber $S^{2n-1}$ of $F$, and so we can connect them with a path on $S^{2n-1}$ tangent to $\mathcal{D}$ which has $g_t$-length less than $C(1-nt)$, as before. Therefore,
\begin{equation}\label{six}
d_t(x,G(F(x)))\leq C(1-nt).
\end{equation}

Combining \eqref{one}, \eqref{two}, \eqref{three}, \eqref{four}, \eqref{five} and \eqref{six} gives the desired Gromov-Hausdorff convergence.
\end{proof}

\begin{remark}
One can also realize these Hopf manifolds and the metric $\omega_H$ as a special case of a construction of Calabi-Eckmann \cite{CE, BV, Ti}.
Indeed consider the holomorphic $\mathbb{C}$-action on $(\mathbb{C}^p\backslash \{0\})\times (\mathbb{C}^q\backslash \{0\})$ given by
$$t\cdot(z,w)=(e^t z_1,\dots e^t z_p, e^{\beta_1 t}w_1,\dots,e^{\beta_q t}w_q),$$
where $t\in\mathbb{C}$, $z\in \mathbb{C}^p\backslash \{0\}, w\in \mathbb{C}^q\backslash \{0\}$ ($0<p\leq q$), and $\beta_1,\dots,\beta_q$ are complex numbers with $\mathrm{Im}\beta_1=\dots=\mathrm{Im}\beta_q\neq 0$.
The quotient $M_{p,q}$ is a complex manifold diffeomorphic to $S^{2p-1}\times S^{2q-1}$.

In the special case when $\beta_1=\dots=\beta_q$ we recover exactly the Calabi-Eckmann manifolds, which
are elliptic bundles over $\mathbb{CP}^{p-1}\times\mathbb{CP}^{q-1}$.

If $p=1$ we have that $M_{1,q}$ is biholomorphic to the Hopf manifold $H=(\mathbb{C}^q\backslash\{0\})/\sim$
where
$(w_1,\dots,w_q)\sim (\alpha_1w_1,\dots,\alpha_qw_q),$
and $\alpha_j=e^{2\pi \mn \beta_j}$ (note that $|e^{2\pi \mn \beta_1}|=\dots=|e^{2\pi \mn \beta_q}|\neq 1$). Indeed every point $(z,w)\in \mathbb{C}^*\times (\mathbb{C}^q\backslash \{0\})$ is in the same $\mathbb{C}$-orbit as $(1, w')$ (just pick $t=-\log z$, for any branch of the complex log), which we can view as a point in $\mathbb{C}^q\backslash \{0\}$. Now note that $t\cdot(1,w')=(1,w'')$ iff $t\in 2\pi \mn \mathbb{Z}$, iff
$w''_j=w'_j e^{2\pi \mn \beta_j \ell}$ for some $\ell\in\mathbb{Z}$ and for all $1\leq j\leq q$. Therefore we have constructed a holomorphic bijection from $M_{1,q}$ to the Hopf manifold $H$ . The inverse biholomorphism  $\Psi:H\to M_{1,q}$ is simply induced
by the map $w\mapsto (1,w)$.

There is a natural Hermitian metric on $M_{p,q}$ given by
$$\omega_0=\sum_{i,j=1}^p \frac{\delta_{ij}}{|z|^2}\mn dz_i\wedge d\ov{z}_j+
\sum_{k,\ell=1}^q\frac{\delta_{k\ell}}{|w|^2}\mn dw_k\wedge d\ov{w}_\ell,$$
and in the case when $p=1$ we have $\Psi^*\omega_0=\omega_H$.

\end{remark}

\section{The Inoue surfaces $S_M$}\label{sectinoue1}
Inoue surfaces were discovered in \cite{In}, and can be characterized as surfaces of class VII with second Betti number zero and with no holomorphic curves \cite{Bog, LYZ, T0}. They form three families, $S_M, S^{+}_{N,p,q,r;{\bf t}}$ and $S^{-}_{N,p,q,r},$ which we will treat separately in the following three sections.

From \cite[Theorem 1.5]{TW3} we know that on any Inoue surface the Chern-Ricci flow starting at any Gauduchon metric $\omega_0$ has a solution $\omega(t)$ for all $t\geq 0$, with volume that grows linearly in $t$. We will consider explicit metrics $\omega_0$ and determine the Gromov-Hausdorff limit of the rescaled metrics $\frac{\omega(t)}{t}$.

In this section we study the Inoue surfaces $S_M$, whose construction from \cite{In} we now recall.
Let $H=\{z\in\mathbb{C}\ |\ \mathrm{Im}z>0\}$ be the upper half plane, and consider the product $H\times\mathbb{C}$.
Let $M\in SL(3,\mathbb{Z})$ be a matrix with one real eigenvalue $\alpha>1$ and two complex conjugate eigenvalues $\beta\neq \ov{\beta}$ (so that $\alpha|\beta|^2=1$). The real number $\alpha$ is necessarily irrational. Let $(a_1,a_2,a_3)$ be a real eigenvector for $M$ with eigenvalue $\alpha$ and $(b_1,b_2,b_3)$ be an eigenvector with eigenvalue $\beta$. Note that since $(a_1,a_2,a_3), (b_1,b_2,b_3)$ and $(\ov{b_1},\ov{b_2},\ov{b_3})$ are
$\mathbb{C}$-linearly independent, it follows that $(a_1, \mathrm{Re} b_1, \mathrm{Im} b_1), (a_2, \mathrm{Re} b_2, \mathrm{Im} b_2)$ and $(a_3, \mathrm{Re} b_3, \mathrm{Im} b_3)$ are $\mathbb{R}$-linearly independent.
Let $\Gamma$ be the group of automorphisms of $H\times\mathbb{C}$ generated by
$$f_0(z,w)=(\alpha z,\beta w),$$
$$f_j(z,w)=(z+a_j,w+b_j), 1\leq j\leq 3.$$
Then $S_M=(H\times\mathbb{C})/\Gamma$ is an Inoue surface. Consider the Tricerri metric \cite{Tr}
$$\omega_0=\frac{1}{y^2}\mn dz\wedge d\ov{z}+y \mn dw\wedge d\ov{w},$$
where $z=x+\mn y$. It is easy to check that $\omega_0$ is $\Gamma$-invariant and descends to a Hermitian metric on $S_M$
which is Gauduchon (because $y$ is a harmonic function).
Let now
$\omega(t)=\omega_0-t\Ric(\omega_0).$
We calculate
$$\Ric(\omega_0)=-\mn\de\db\log\det((g_0)_{i\ov{j}})=\mn\de\db\log y=-\frac{1}{4y^2}\mn dz\wedge d\ov{z},$$
and so
\begin{equation}\label{evolve}
\omega(t)=\left(1+\frac{t}{4}\right)\frac{1}{y^2}\mn dz\wedge d\ov{z}+y \mn dw\wedge d\ov{w},
\end{equation}
which is a Gauduchon metric for all $t>0$. It satisfies the Chern-Ricci flow, because
$$\det(g_{i\ov{j}}(t))=\left(1+\frac{t}{4}\right)\frac{1}{y}=\left(1+\frac{t}{4}\right)\det((g_0)_{i\ov{j}}),$$
and so $\Ric(\omega(t))=\Ric(\omega_0)=-\frac{\de}{\de t}\omega(t).$

If we renormalize the metrics by dividing by $t$ and we let $t$ go to infinity we get
$$\frac{\omega(t)}{t}\to \omega_\infty=\frac{1}{4y^2}\mn dz\wedge d\ov{z},$$
smoothly on $H\times\mathbb{C}$ (and on $S_M$). The limit degenerate metric $\omega_\infty$ is simply the pullback of one half of the Poincar\'e metric from $H$, $\omega_{\mathrm{KE}}=\frac{dx^2+dy^2}{y^2}$, so in particular it is closed (unlike in the case
of the Hopf surface). This degenerate metric has appeared for the first time in this context in \cite{HL}.

\begin{theorem}\label{inoue}
As $t$ approaches $+\infty$ we have that $\left(S_M, \frac{\omega(t)}{t}\right)\overset{GH}{\to} (S^1,d)$, where $d$ is the distance function on the circle $S^1\subset\mathbb{R}^2$ with radius $\frac{\log \alpha}{2\sqrt{2}\pi}$.
\end{theorem}

To calculate the Gromov-Hausdorff limit of $\left(S_M, \frac{\omega(t)}{t}\right)$ we need to understand the topology of $S_M$ a bit better. The key observation, due to Inoue, is that $S_M$ is a $T^3$-bundle over $S^1$. Indeed, if we consider the subgroup $\Gamma'\subset\Gamma$ generated by $f_1, f_2, f_3$, then $\Gamma'$ is isomorphic to $\mathbb{Z}^3$ (because of the linear independence property mentioned above) and it acts properly discontinuously and freely on $H\times\mathbb{C}$, with quotient the product $X=T^3\times \mathbb{R}^+$ (since the numbers $a_j$ are real). The projection
$\pi:X\to\mathbb{R}^+$ is induced by $(z,w)\mapsto \mathrm{Im}z$.

Since $\alpha a_j=\sum_k m_{jk} a_k$ and $\beta b_j=\sum_k m_{jk} b_k$ where $M=(m_{jk})$ and $m_{jk}\in\mathbb{Z}$,
we see that $f_0$ descends to a map $X\to X$.
We have that $S_M=X/\langle f_0\rangle$, and since $\alpha\in\mathbb{R}$ we see that $f_0$ maps the torus fiber $T_y=\pi^{-1}(y)$ to the torus fiber $T_{\alpha y}=\pi^{-1}(\alpha y)$. In particular, $f_0$ induces a diffeomorphism of the $3$-torus
$\psi:T_1\to T_{\alpha}$ and we have that $S_M$ is diffeomorphic to the quotient space $([1,\alpha]\times T^3)/\sim,$
where $(1,p)\sim (\alpha, \psi(p)),$ which is a $T^3$-bundle over $S^1$ (recall that $\alpha>1$). We will still call
$\pi:S_M\to S^1$ the projection map.

The kernel of $\omega_\infty$ on $H\times\mathbb{C}$ is the integrable distribution $\mathcal{D}=\mathrm{Span}_{\mathbb{C}}\left(\frac{\de}{\de w}\right)$, whose leaves are of the
form $\mathcal{L}_{z_0}=\{(z_0,w)\ |\ w\in\mathbb{C}\}\subset H\times\mathbb{C}.$ We wish to determine the images
of these leaves when projected to $S_M$.

\begin{lemma}\label{dense} If we call $P:H\times\mathbb{C}\to S_M$ the projection, then for any $z_0\in H$ the image
$P(\mathcal{L}_{z_0})$ is dense inside the $T^3$ fiber $T_0=\pi^{-1}(\mathrm{Im}z_0)\subset S_M$.
\end{lemma}
\begin{proof}
It is clear that $P(\mathcal{L}_{z_0})\subset T_0$. Obviously $P(\mathcal{L}_{z_0})$ is just a leaf of the foliation $\mathcal{D}$ on $S_M$. If
$P(\mathcal{L}_{z_0})$ were closed in $S_M$ then it would be a complex curve in $S_M$, but Inoue \cite{In} has shown that there are no such curves.
Therefore $P(\mathcal{L}_{z_0})$ cannot be closed, and since it is contained in the $3$-torus $T_0$, it must be dense in $T_0$ (the closure of any leaf of a linear foliation of a torus is always a torus itself).
\end{proof}
An alternative direct proof of this lemma can be given along the lines of \cite[Proposition V.19.1]{bhpv}.

\begin{proof}[Proof of Theorem \ref{inoue}]
For $t>0$ call $d_t$ the distance function on $S_M$ induced by the metric $\frac{\omega(t)}{t}$, and
let $L_t(\gamma)$ denote the length of a curve $\gamma$ with respect to $\frac{\omega(t)}{t}$.
Similarly, $d_0$ and $L_0(\gamma)$ are defined using $\omega_0$, and we will denote by
$L_\infty(\gamma)$ the length of $\gamma$ computed using the degenerate metric $\omega_\infty$.

On the circle $S^1$ we put the metric $\frac{1}{2}(d\log r)^2=\frac{(dr)^2}{2r^2}$, and denote by $d$ its distance function.
It is isometric to the standard metric on $S^1\subset \mathbb{R}^2$ with radius $\frac{\log \alpha}{2\sqrt{2}\pi}$.

For any $\ve>0$ fixed, we will show that for $t$ sufficiently large the Gromov-Hausdorff distance between
$(S_M, d_t)$ and $(S^1,d)$ is less than $3\ve$.

We regard $S^1$ as $[1,\alpha]/(1\sim \alpha)$ and $S_M$ as $([1,\alpha]\times T^3)/\sim,$
where $(1,p)\sim (\alpha, \psi(p)),$ as before.
Call $F:S_M\to S^1$ the projection of the $T^3$-bundle, and let
$G:S^1\to S_M$ be the discontinuous map induced by the map $\phi:[1,\alpha]\to [1,\alpha]\times T^3$
given by $\phi(x)=(x,p_0)$ for for $x\in[1,\alpha)$ and $\phi(\alpha)=(1,p_0)$, where $p_0\in T^3$
is a fixed basepoint.

Clearly we have $F\circ G=\mathrm{Id}$, while $G\circ F$ is a fiber-preserving discontinuous map of $S_M$.
In particular for any $a\in S^1$ we have trivially
\begin{equation}\label{one1}
d(a,F(G(a)))=0.
\end{equation}
First of all observe that from \eqref{evolve} there is a constant $C_0$ so that
for all $t\geq 1$ and for any curve $\gamma$ in $S_M$ we have
$L_t(\gamma)\leq C_0 L_0(\gamma)$.

Second, recall that $\mathcal{D}=\ker\omega_\infty,$ so from \eqref{evolve} again we
see that if $\gamma$ is a curve in $S_M$ with tangent vector
always in $\mathcal{D}$, then $L_t(\gamma)\leq \frac{C_0}{\sqrt{t}} L_0(\gamma)$.

Third, let $p,q$ be any two points in $S_M$ on a same $T^3$-fiber, i.e. $F(p)=F(q)$, and
pick any $(z_0,w_0)\in H\times \mathbb{C}$ such that $P(z_0,w_0)=p$, so that the image of
the leaf $P(\mathcal{L}_{z_0})$ passes through $p$.
Thanks to Lemma \ref{dense} we know that $P(\mathcal{L}_{z_0})$ is dense in the $T^3$-fiber,
and so there is a connected compact set $K\subset \mathcal{L}_{z_0}$ such that every point in this $T^3$-fiber
has $d_0$-distance less than $\frac{\ve}{2C_0}$ to $P(K)$. On the other hand, every point in $P(K)$ can be
joined to $p$ with a curve $\gamma$ in the $T^3$-fiber with tangent vector in $\mathcal{D}$.
Therefore, for any such $\gamma$ we have $L_t(\gamma)\leq\frac{C}{\sqrt{t}}$,
with a uniform constant $C$ independent of $\gamma$ (it depends only on $K$).

It follows that there is a point in $P(K)$ which can be joined to $p$ by a curve $\gamma_1$ tangent to $\mathcal{D}$
and to $q$ by a curve $\gamma_2$ with $L_0(\gamma_2)\leq \frac{\ve}{2C_0}$.
Concatenating $\gamma_1$ and $\gamma_2$ we see that
\begin{equation}\label{fiber}
d_t(p,q)\leq L_t(\gamma_1)+L_t(\gamma_2)
\leq\frac{C}{\sqrt{t}} + C_0 L_0(\gamma_2)
\leq \frac{C}{\sqrt{t}} + \frac{\ve}{2}\leq \ve,
\end{equation}
if $t$ is large enough.

Let now $p,q$ be any two points in $S_M$, with $F(p)=a, F(q)=b$ where $1\leq a,b<\alpha$, and we can assume that
$a\leq b$.
Then $p$ and the point $(a,p_0)$ (rather its equivalence class in $S_M$) belong to the same $T^3$ fiber,
as do $q$ and $(b,p_0)$, so from \eqref{fiber} we get
\begin{equation}\label{fiber2}
d_t(p, (a,p_0))\leq \ve,\quad d_t(q, (b,p_0))\leq \ve.
\end{equation}
We then join $(a,p_0)$ to $(b,p_0)$ via the image in $S_M$ of the curve
$\gamma(s)=(s,p_0)$, $a\leq s\leq b$. The point $\gamma(s)$ has a lift to $H\times\mathbb{C}$
with imaginary part equal to $s$, so the tangent vector to $\gamma(s)$ is $\frac{\de}{\de y}=-\mn(\frac{\de}{\de \ov{z}}-\frac{\de}{\de z})$
so from \eqref{evolve} we have
$$\left||\gamma'|^2_{\frac{\omega(t)}{t}}-|\gamma'|^2_{\omega_\infty}\right|\leq \frac{C}{t},$$
and so $|L_t(\gamma)-L_\infty(\gamma)|\leq \frac{C}{\sqrt{t}}$.
But for this curve $\gamma$ we have
$$L_\infty(\gamma)=\int_a^b |\gamma'(s)|_{\omega_\infty}ds= \frac{1}{\sqrt{2}}\int_{a}^b \frac{ds}{s}=\frac{1}{\sqrt{2}}\log(b/a)
=d(F(p),F(q)),$$
and so combining this with \eqref{fiber2} we have proved that
$$d_t(p,q)\leq 2\ve+L_t(\gamma)\leq \frac{C}{\sqrt{t}} + 2\ve+d(F(p),F(q)),$$
for a constant $C$ independent of $p,q$, so if $t$ is large we get
\begin{equation}\label{two1}
d_t(p,q)\leq d(F(p),F(q))+3\ve,
\end{equation}
and so also
\begin{equation}\label{three1}
d_t(G(a),G(b))\leq d(a,b)+3\ve,
\end{equation}
for all $a,b\in S^1$.

Note that from \eqref{fiber} we also have that for any $p\in S_M$ and for all $t$ large
\begin{equation}\label{four1}
d_t(p,G(F(p)))\leq \ve.
\end{equation}

Now take any two points $p,q \in S_M$ and let $\gamma$ be a curve joining $p$ to $q$ with $L_t(\gamma)=d_t(p,q)$.  Then $F(\gamma)$ is a path in $S^1$ between $F(p)$ and $F(q)$.  Write $L_g(F(\gamma))$ for the length of this curve with respect to the metric $g= \frac{1}{2} (d\log y)^2$, where we are using the coordinate $y$ on $S^1$.   We claim that $L_g(F(\gamma))= L_{\infty}(\gamma)$.  Indeed, if $V$ is a tangent vector on $S_M$ we can write locally $V = V_1 \frac{\partial}{\partial x} + V_2 \frac{\partial}{\partial y} + V_3 \frac{\partial}{\partial u} + V_4 \frac{\partial}{\partial v}$ where $w = u+\sqrt{-1}v$.  But $F_* V = V_2 \frac{\partial}{\partial y}$ and from the definition of $\omega_{\infty}$ and $g$ we see that $|F_*V|^2_g = \frac{V^2_2}{2y^2} = |V|^2_{\omega_{\infty}}$.  Applying this with $V= \gamma'$ proves the claim.

Noting that $\omega_{\infty} \le \frac{\omega(t)}{t}$ we have:
\begin{equation} \label{five1}
d(F(p), F(q)) \le L_g (F(\gamma)) = L_{\infty}(\gamma) \le L_t(\gamma) = d_t(p,q).
\end{equation}
For  $a, b \in S^1$, we can apply (\ref{five1}) to $p=G(a)$ and $q=G(b)$ to obtain
\begin{equation}\label{six1}
d(a,b)\leq d_t(G(a),G(b)).
\end{equation}
Combining \eqref{one1}, \eqref{two1}, \eqref{three1}, \eqref{four1}, \eqref{five1} and \eqref{six1} shows that the Gromov-Hausdorff distance between $(S_M, d_t)$ and $(S^1,d)$ can be made less than $3\ve$ if $t$ is large, as required.
\end{proof}

\section{The Inoue surfaces $S^{+}_{N,p,q,r;{\bf t}}$}\label{sectinoue2}

In this section we study the Inoue surfaces $S^{+}_{N,p,q,r;{\bf t}}$, starting from their construction from \cite{In}.
Let $N=(n_{ij}) \in SL(2,\mathbb{Z})$ be a matrix with two real eigenvalues $\alpha>1$ and $\frac{1}{\alpha}$.
Let $(a_1,a_2)$ and $(b_1,b_2)$ be two real eigenvectors for $N$ with eigenvalues $\alpha$ and $\frac{1}{\alpha}$ respectively (again we automatically have that $\alpha$ is irrational).

Fix integers $p,q,r$, with $r\neq 0$, and a complex number ${\bf t}$.
Using $N, a_j,$ $b_j, p,q,r$ one gets two real numbers $(c_1,c_2)$ as solutions of
the linear equation
$$(c_1,c_2)=(c_1,c_2)\cdot N^{t} + (e_1,e_2) + \frac{b_1a_2-b_2a_1}{r}(p,q),$$
where
$$e_i=\frac{1}{2} n_{i1}(n_{i1}-1)a_1b_1 + \frac{1}{2} n_{i2}(n_{i2}-1)a_2b_2 + n_{i1}n_{i2} b_1a_2,
\quad i=1,2.$$
Let $\Gamma$ be the group of automorphisms of $H\times\mathbb{C}$ generated by
$$f_0(z,w)=(\alpha z,w+{\bf t}),$$
$$f_j(z,w)=(z+a_j,w+b_jz+c_j),\quad j=1,2,$$
$$f_3(z,w)=\left(z,w+\frac{b_1a_2-b_2a_1}{r}\right).$$
Then $S^{+}_{N,p,q,r;{\bf t}}=(H\times\mathbb{C})/\Gamma$ is an Inoue surface.

Since $\alpha>1$, we can write $\mathrm{Im}{\bf t}=m\log \alpha$ for some $m\in\mathbb{R}$,
so that ${\bf t}$ is real iff $m=0$. Note that the $(1,0)$-forms on $H\times\mathbb{C}$
$$\frac{1}{y}dz,\quad dw -\frac{v-m\log y}{y}dz$$
(where $z=x+\mn y, w=u+\mn v$) are invariant under the $\Gamma$-action, and so descend to $S^{+}_{N,p,q,r;{\bf t}}$, so we
can define a Hermitian metric $S^{+}_{N,p,q,r;{\bf t}}$
\begin{equation}\label{metric}
\begin{split}
\omega_0&=\mn \left(dw -\frac{v-m\log y}{y}dz\right)\wedge \left(d\ov{w} -\frac{v-m\log y}{y}d\ov{z}\right)\\
&\ \ \ \ +\frac{1}{y^2} \mn dz\wedge d\ov{z}\\
&=\mn dw\wedge d\ov{w}+\left(\frac{1+(v-m\log y)^2}{y^2}\right)\mn dz\wedge d\ov{z} \\
&\ \ \ \ -\frac{v-m\log y}{y} \mn dw\wedge d\ov{z}
-\frac{v-m\log y}{y} \mn dz\wedge d\ov{w}.
\end{split}
\end{equation}
This was discovered by Tricerri \cite{Tr} when $m=0$ and by Vaisman \cite{Va} in general.
The key difference between the cases when ${\bf t}$ is real or not is that when $m=0$
the metric $\omega_0$ is locally conformally K\"ahler, while when $m\neq 0$ it is not, and in fact
a theorem of Belgun \cite[Theorem 7]{Be} shows that the surfaces $S^{+}_{N,p,q,r;{\bf t}}$ with ${\bf t}$ not real do not admit any locally conformally K\"ahler metric.

On the other hand, we can easily check that $\omega_0$ is Gauduchon for any value of $m$:
$$\db\omega_0=-\frac{v-m\log y-m}{2y^2} d\ov{w}\wedge dz\wedge d\ov{z}+\frac{1}{2y} d\ov{w}\wedge dw\wedge d\ov{z},$$
$$\de\db\omega_0=\frac{\sqrt{-1}}{4y^2} dw\wedge d\ov{w}\wedge dz\wedge d\ov{z} + \frac{\sqrt{-1}}{4y^2} dz\wedge d\ov{w}\wedge dw\wedge d\ov{z}=0.$$
Let now
$\omega(t)=\omega_0-t\Ric(\omega_0).$
We calculate
$$\det((g_0)_{i\ov{j}})=\frac{1}{y^2},$$
$$\Ric(\omega_0)=-\mn\de\db\log\det((g_0)_{i\ov{j}})=2\mn\de\db\log y=-\frac{1}{2y^2}\mn dz\wedge d\ov{z},$$
and so
\begin{equation}\label{evolve2}
\begin{split}
\omega(t)&=\mn dw\wedge d\ov{w}+\left(\frac{1+(v-m\log y)^2+t/2}{y^2}\right)\mn dz\wedge d\ov{z} \\
&-\frac{v-m\log y}{y} \mn dw\wedge d\ov{z}-\frac{v-m\log y}{y} \mn dz\wedge d\ov{w},
\end{split}
\end{equation}
which is a Gauduchon metric for all $t>0$. It satisfies the Chern-Ricci flow, because
$$\det(g_{i\ov{j}}(t))=\left(1+\frac{t}{2}\right)\frac{1}{y^2}=\left(1+\frac{t}{2}\right)\det((g_0)_{i\ov{j}}),$$
and so $\Ric(\omega(t))=\Ric(\omega_0)=-\frac{\de}{\de t}\omega(t).$

If we renormalize the metrics by dividing by $t$ and we let $t$ go to infinity we obtain
$$\frac{\omega(t)}{t}\to \frac{1}{2y^2}\mn dz\wedge d\ov{z},$$
smoothly on $H\times\mathbb{C}$. The limit degenerate metric is simply the pullback of the Poincar\'e metric from $H$.

\begin{theorem}\label{inoue2}
As $t$ approaches infinity we have that $\left(S^{+}_{N,p,q,r;{\bf t}}, \frac{\omega(t)}{t}\right)\overset{GH}{\to} (S^1,d)$, where $d$ is the distance function on the circle $S^1\subset\mathbb{R}^2$ with radius $\frac{\log \alpha}{2\pi}$.
\end{theorem}

Again we need to understand the topology of $S^{+}_{N,p,q,r;{\bf t}}$. As remarked by Inoue \cite{In} $S^{+}_{N,p,q,r;{\bf t}}$ is diffeomorphic to a bundle over $S^1$ with fiber a compact $3$-manifold $X$. Indeed, if we consider the subgroup $\Gamma'\subset\Gamma$ generated by $f_1, f_2, f_3$, then for each fixed $y=\mathrm{Im}z$ the group $\Gamma'$ acts on
$\{(x,y,w)\ |\ x\in\mathbb{R}, w\in\mathbb{C}\}\cong \mathbb{R}^3$ properly discontinuously and freely with quotient
a $3$-manifold $X_y$. For different values of $y$ they are all diffeomorphic to a fixed manifold $X$. Then, as in the case of $S_M$,
we can also consider $\Gamma'$ acting on the whole of $H\times\mathbb{C}$, and the
quotient is diffeomorphic to the product $X\times \mathbb{R}^+$, with the projection $\pi$
to $\mathbb{R}^+$ induced by $(z,w)\mapsto \mathrm{Im}z$ and with
$X_y=\pi^{-1}(y)$.
Then again $f_0$ descends to a map $X\times \mathbb{R}^+\to X\times \mathbb{R}^+$, because $f_0$ lies in the normalizer of $\Gamma'$ \cite[p.276]{In}.
We have that $S^{+}_{N,p,q,r;{\bf t}}=(X\times\mathbb{R}^+)/\langle f_0\rangle$.
Since $\alpha\in\mathbb{R}$ we see that $f_0$ maps the fiber $X_1$ to the fiber $X_\alpha$, and
so it induces a diffeomorphism $\psi$ of $X$ such that $S^{+}_{N,p,q,r;{\bf t}}$ is diffeomorphic to the quotient space $([1,\alpha]\times X)/\sim,$
where $(1,p)\sim (\alpha, \psi(p)),$ which is an $X$-bundle over $S^1$. We will still call
$\pi:S^{+}_{N,p,q,r;{\bf t}}\to S^1$ the projection map.

The kernel of $\omega_\infty$ on $H\times\mathbb{C}$ is the integrable distribution (i.e. foliation) $\mathcal{D}=\mathrm{Span}_{\mathbb{C}}\left(\frac{\de}{\de w}\right)$, whose leaves are of the
form $\mathcal{L}_{z_0}=\{(z_0,w)\ |\ w\in\mathbb{C}\}\subset H\times\mathbb{C}.$ We wish to determine the images
of these leaves when projected to $S^{+}_{N,p,q,r;{\bf t}}$. The main idea in the following Lemma comes from \cite{Br}.

\begin{lemma}\label{dense2} If we call $P:H\times\mathbb{C}\to S^{+}_{N,p,q,r;{\bf t}}$ the projection, then for any $z_0\in H$ the image
$P(\mathcal{L}_{z_0})$ is dense inside the fiber $X_0=\pi^{-1}(\mathrm{Im}z_0)\subset S_M$.
\end{lemma}
\begin{proof}
The proof is similar to the one of Lemma \ref{dense}.
It is clear that $P(\mathcal{L}_{z_0})\subset X_0$.
Obviously $P(\mathcal{L}_{z_0})$ is just a leaf of the foliation $\mathcal{D}$ on $S^{+}_{N,p,q,r;{\bf t}}$. No such leaf can be closed in $S^{+}_{N,p,q,r;{\bf t}}$, since otherwise it would be a complex curve in $S^{+}_{N,p,q,r;{\bf t}}$, contradicting the fact that there exist no curves \cite{In}.  Since $\mathcal{D}\subset TX_0$, the same is true for the induced foliation $\mathcal{D}|_{X_0}$.

Consider now the $1$-form $dx$, where $x=\mathrm{Re}z$. It is invariant under $\Gamma'$, so it descends to a $1$-form on $X\times\mathbb{R}^+$, and since $f_0$ maps the fiber $X_0$ to a different fiber, it follows that $dx$ is a well-defined closed $1$-form in an open neighborhood of $X_0$. When restricted to $X_0$, the $1$-form $dx$ defines the foliation $\mathcal{D}|_{X_0}$, in the sense that $\ker dx=\mathcal{D}|_{X_0}$. Since we have just seen that no leaf of $\mathcal{D}|_{X_0}$ is closed in $X_0$, we can apply the general theory of foliations defined by closed $1$-forms \cite[4.3, p.46]{Go} and conclude that every leaf of $\mathcal{D}|_{X_0}$ is dense in $X_0$. In particular this is the case for $P(\mathcal{L}_{z_0})$.
\end{proof}

\begin{proof}[Proof of Theorem \ref{inoue2}]
With these preliminaries in place, the proof is almost identical to the proof of Theorem \ref{inoue}, and therefore we only indicate the necessary modifications.

For $t>0$ call $d_t$ the distance function on $S^{+}_{N,p,q,r;{\bf t}}$ induced by the metric $\frac{\omega(t)}{t}$, and
let $L_t(\gamma)$ denote the length of a curve $\gamma$ with respect to $\frac{\omega(t)}{t}$.
Similarly, $d_0$ and $L_0(\gamma)$ are defined using $\omega_0$, and we will denote by
$L_\infty(\gamma)$ the length of $\gamma$ computed using the degenerate metric $\omega_\infty$.

On the circle $S^1$ we put the metric $(d\log r)^2=\frac{(dr)^2}{r^2}$, and denote by $d$ its distance function, which is isometric to the standard metric on $S^1\subset \mathbb{R}^2$ with radius $\frac{\log \alpha}{2\pi}$.
For any $\ve>0$ fixed, we will show that for $t$ sufficiently large the Gromov-Hausdorff distance between
$(S^{+}_{N,p,q,r;{\bf t}}, d_t)$ and $(S^1,d)$ is less than $3\ve$.

We regard $S^1$ as $[1,\alpha]/(1\sim \alpha)$ and $S^{+}_{N,p,q,r;{\bf t}}$ as $([1,\alpha]\times X)/\sim,$
where $(1,p)\sim (\alpha, \psi(p)),$ as before.
Call $F:S^{+}_{N,p,q,r;{\bf t}}\to S^1$ the projection of the $X$-bundle, and let
$G:S^1\to S^{+}_{N,p,q,r;{\bf t}}$ be the discontinuous map induced by the map $\phi:[1,\alpha]\to [1,\alpha]\times X$
given by $\phi(x)=(x,p_0)$ for for $x\in[1,\alpha)$ and $\phi(\alpha)=(1,p_0)$, where $p_0\in X$
is a fixed basepoint.

Clearly we have $F\circ G=\mathrm{Id}$, while $G\circ F$ is a fiber-preserving discontinuous map of $S^{+}_{N,p,q,r;{\bf t}}$.
In particular for any $a\in S^1$ we have trivially
\begin{equation}\label{one2}
d(a,F(G(a)))=0.
\end{equation}
Exactly as in Theorem \ref{inoue} we prove that for any two points $p,q\in S^{+}_{N,p,q,r;{\bf t}}$ on the same $X$-fiber
we have $d_t(p,q)\leq \ve,$ if $t$ is large enough. From this we deduce that for any two points $p,q$ and for all $t$ large we have
\begin{equation}\label{two2}
d_t(p,q)\leq d(F(p),F(q))+3\ve,
\end{equation}
and so also
\begin{equation}\label{three2}
d_t(G(a),G(b))\leq d(a,b)+3\ve,
\end{equation}
for all $a,b\in S^1$ and
\begin{equation}\label{four2}
d_t(p,G(F(p)))\leq \ve.
\end{equation}

Take now any two points $p,q\in S^{+}_{N,p,q,r;{\bf t}}$ and call $\gamma$ a curve
joining $p$ and $q$ with $L_t(\gamma)=d_t(p,q)$.  Arguing as in the proof of Theorem \ref{inoue}, using the fact that $\omega_{\infty} \le \frac{\omega(t)}{t}$,
\begin{equation} \label{five2}
d(F(p), F(q)) \le L_g(F(\gamma)) = L_{\infty}(\gamma) \le L_t(\gamma) = d_t(p,q),
\end{equation}
where we are writing $L_g(F(\gamma))$ for the length of the path $F(\gamma)$ on $S^1$ between $F(p)$ and $F(q)$ with respect to the metric $g = (d\log y)^2$.
Hence also
\begin{equation}\label{six2}
d(a,b)\leq d_t(G(a),G(b)),
\end{equation}
for all $a,b\in S^1$.
Combining \eqref{one2}, \eqref{two2}, \eqref{three2}, \eqref{four2}, \eqref{five2} and \eqref{six2} shows that the Gromov-Hausdorff distance between $(S^{+}_{N,p,q,r;{\bf t}}, d_t)$ and $(S^1,d)$ can be made less than $3\ve$ if $t$ is large, as required.
\end{proof}

\section{The Inoue surfaces $S^{-}_{N,p,q,r}$}\label{sectinoue3}

The last class of Inoue surfaces are $S^{-}_{N,p,q,r}$, defined as follows.
Let $N=(n_{ij})\in GL(2,\mathbb{Z})$ be a matrix with $\det N=-1$ and with two real eigenvalues $\alpha>1$ and $-\frac{1}{\alpha}$.
Let $(a_1,a_2)$ and $(b_1,b_2)$ be two real eigenvector for $N$ with eigenvalues $\alpha$ and $-\frac{1}{\alpha}$ respectively. Fix integers $p,q,r$, with $r\neq 0$.
Define two real numbers $(c_1,c_2)$ as solutions of
the linear equation
$$-(c_1,c_2)=(c_1,c_2)\cdot N^{t} + (e_1,e_2) + \frac{b_1a_2-b_2a_1}{r}(p,q),$$
where $e_i$ are the same as for the surfaces $S^{+}_{N,p,q,r;{\bf t}}$.

Let $\Gamma$ be the group of automorphisms of $H\times\mathbb{C}$ generated by
$$f_0(z,w)=(\alpha z,-w),$$
$$f_j(z,w)=(z+a_j,w+b_jz+c_j),\quad  j=1,2,$$
$$f_3(z,w)=\left(z,w+\frac{b_1a_2-b_2a_1}{r}\right).$$
Then $S^{-}_{N,p,q,r}=(H\times\mathbb{C})/\Gamma$ is an Inoue surface.

As noticed by Tricerri \cite{Tr}, the exact same formula as in the case of $S^{+}_{N,p,q,r;{\bf t}},$ ${\bf t}$ real,
(i.e. \eqref{metric} with $m=0$)
gives a Gauduchon metric on $S^{-}_{N,p,q,r}$ too. The discussion of the smooth limit of the Chern-Ricci flow is identical.

Every surface $S^{-}_{N,p,q,r}$ has as an unramified double cover an Inoue surface $S^{+}_{N^2,p',q',r;0}$ (for suitable integers $p', q'$), so we can pull back everything upstairs and reduce to the previous section. Indeed, we have the involution of $S^{+}_{N^2,p',q',r;0}$
$$\iota(z,w)=(\alpha z,-w),$$
which satisfies $\iota^2=\mathrm{Id}$ and $S^{+}_{N^2,p',q',r;0}/\iota=S^{-}_{N,p,q,r}$.
We will denote by $p:S^{+}_{N^2,p',q',r;0}\to S^{-}_{N,p,q,r}$ the quotient map.
The projection $F:S^{+}_{N^2,p',q',r;0}\to S^1=\mathbb{R}^+/(x\sim \alpha^2 x)$ satisfies
$F(\iota(x))=F(x)$ for all $x\in S^{+}_{N^2,p',q',r;0}$.

\begin{theorem}\label{inoue3}
As $t$ approaches $+\infty$ we have that $\left(S^{-}_{N,p,q,r}, \frac{\omega(t)}{t}\right)\overset{GH}{\to} (S^1,d)$, where $d$ is the distance function on the circle $S^1\subset\mathbb{R}^2$ with radius $\frac{\log \alpha}{\pi}$.
\end{theorem}

\begin{proof} We pull back the metrics $\omega_0, \omega(t)$ via $p$ to $S^{+}_{N^2,p',q',r;0}$,
and obtain the same metrics as in Theorem \ref{inoue2}, with $\iota$ acting on $S^{+}_{N^2,p',q',r;0}$ as an isometry of $\frac{p^*\omega(t)}{t}$. From Theorem \ref{inoue2} we see that $\left(S^{+}_{N^2,p',q',r;0}, \frac{p^*\omega(t)}{t}\right)$ converges in Gromov-Hausdorff to $S^1=\mathbb{R}^+/(x\sim \alpha^2 x)$
with the metric $(d\log x)^2$ (isometric to the standard metric on $S^1\subset\mathbb{R}^2$ with radius $\frac{\log \alpha}{\pi}$).

In fact this convergence happens also in the equivariant Gromov-Hausdorff sense, where the group acting on $S^1$ is the trivial group while the group acting on $S^{+}_{N^2,p',q',r;0}$ is the group of order two generated by $\iota$.
Indeed it is immediate to check this from the definition of equivariant Gromov-Hausdorff distance (see Section \ref{secthopf}), using the same maps $F:S^{+}_{N^2,p',q',r;0}\to S^1$ and $G:S^1\to S^{+}_{N^2,p',q',r;0}$ from the proof of Theorem \ref{inoue2} (where the maps from between the trivial group and the group of order two are the obvious ones).

Then \cite[Theorem 2.1]{Fuk} or \cite[Lemma 1.5.4]{Ro} imply that $\left(S^{-}_{N,p,q,r}, \frac{\omega(t)}{t}\right)\overset{GH}{\to} (S^1,d)$, where $d$ is the distance function on the circle $S^1\subset\mathbb{R}^2$ with radius $\frac{\log \alpha}{\pi}$.
\end{proof}

Theorems \ref{inoue}, \ref{inoue2} and \ref{inoue3} together complete the proof of part (b) of Theorem \ref{surfaces}.

\section{Non-K\"ahler properly elliptic surfaces}\label{sectelliptic}
Recall that a non-K\"ahler properly elliptic surface is by definition a compact complex surface $S$
with $b_1(S)$ odd and with Kodaira dimension $\kappa(S)=1$ which admits an elliptic fibration $\pi:S\to C$
to a smooth compact curve $C$. Throughout this section we will always assume that $S$ is minimal.
Kodaira \cite[Theorem 28]{Ko} has shown that the universal cover of $S$ is $H\times\mathbb{C}$.
It is also known (see for example \cite[Lemmas 1, 2]{Bri} or \cite[Theorem 7.4]{Wa}) that there is always a finite unramified covering $S'\to S$ which is also a minimal properly elliptic surface $\pi':S'\to C'$ and $\pi'$ is an elliptic fiber bundle with
$g(C')\geq 2$ (the curve $C'$ is a finite cover of $C$ ramified at the images of the multiple fibers of $\pi$).

Let us first assume that we are in this situation, so that $\pi:S\to C$ is an elliptic fiber bundle with fiber $E$, with $g(C)\geq 2$, with $S$ minimal, non-K\"ahler and $\kappa(S)=1$. It will be more convenient for us to work with $H\times\mathbb{C}^*$, so we define
$$h:H\times\mathbb{C}\to H\times\mathbb{C}^*, \quad h(z,z')=(z, e^{-\frac{z'}{2}}),$$
which is a holomorphic covering map, and we will write $(z,w)$ for the coordinates on $H\times\mathbb{C}^*$.
A theorem of Maehara \cite{Ma} (see also \cite[Lemma 5.6]{IKO}, \cite[Theorem 7.4]{Wa} and \cite[Proposition 2]{Be}) shows that there exist $\Gamma\subset SL(2,\mathbb{R})$ a discrete subgroup with
$H/\Gamma=C$, together with a complex number $\alpha\in\mathbb{C}^*$ with $|\alpha|\neq 1$ and $\mathbb{C}^*/\langle \alpha\rangle=E$ and together with a character $\chi:\Gamma\to\mathbb{C}^*$ (i.e. a group homomorphism) such that
$S$ is biholomorphic to the quotient of $H\times\mathbb{C}^*$ by the
$\Gamma\times\mathbb{Z}$-action defined by
$$\left( \left(\begin{array}{cc} a & b  \\ c & d \end{array}\right),n\right)\cdot(z,w)=\left( \frac{az+b}{cz+d}, (cz+d)\cdot w\cdot \alpha^n\cdot \chi\left(\begin{array}{cc} a & b  \\ c & d \end{array}\right) \right),$$
and the map $\pi:S\to C$ is induced by the projection $H\times\mathbb{C}^*\to H$.

The reader can check that the forms on $H\times\mathbb{C}^*$:
$$-\frac{2}{w}dw +\frac{\mn}{y}dz, \quad \frac{1}{y^2}dz\wedge d\ov{z},$$
(where $z=x+\mn y$) are invariant under the $\Gamma\times\mathbb{Z}$-action.
Therefore they descend to $S$ and we
can define a Hermitian metric on $S$ by
\begin{equation}\label{metric2}
\begin{split}
\omega_0&=\mn \left( -\frac{2}{w}dw +\frac{\mn}{y}dz \right)\wedge \left( -\frac{2}{\ov{w}}d\ov{w} -\frac{\mn}{y}d\ov{z} \right)+\frac{1}{y^2} \mn dz\wedge d\ov{z}\\
&=\frac{4}{|w|^2}\mn dw\wedge d\ov{w}+\frac{2}{y^2}\mn dz\wedge d\ov{z} \\
&+\frac{2\mn}{yw} \mn dw\wedge d\ov{z}
-\frac{2\mn}{y\ov{w}} \mn dz\wedge d\ov{w}.
\end{split}
\end{equation}
This metric was discovered by Vaisman \cite{Va} (he wrote down its pullback $h^*\omega_0$ on
$H\times\mathbb{C}$) and it is Gauduchon:
$$\db\omega_0=-\frac{\mn}{y^2\ov{w}} d\ov{z}\wedge dz\wedge d\ov{w},\quad \de\db\omega_0=0.$$
Let now
$\omega(t)=\omega_0-t\Ric(\omega_0).$
We calculate
$$\det((g_0)_{i\ov{j}})=\frac{4}{y^2|w|^2},$$
and
$$\Ric(\omega_0)=-\mn\de\db\log\det((g_0)_{i\ov{j}})=2\mn\de\db\log y=-\frac{1}{2y^2}\mn dz\wedge d\ov{z},$$
and so
\[\begin{split}
\omega(t)&=\frac{4}{|w|^2}\mn dw\wedge d\ov{w}+\left(\frac{2+t/2}{y^2}\right)\mn dz\wedge d\ov{z} \\
&+\frac{2\mn}{yw} \mn dw\wedge d\ov{z}
-\frac{2\mn}{y\ov{w}} \mn dz\wedge d\ov{w},
\end{split}\]
which is a Gauduchon metric for all $t>0$. It satisfies the Chern-Ricci flow, because
$$\det(g_{i\ov{j}}(t))=\left(1+\frac{t}{2}\right)\frac{4}{y^2|w|^2}=\left(1+\frac{t}{2}\right)\det((g_0)_{i\ov{j}}),$$
and so $\Ric(\omega(t))=\Ric(\omega_0)=-\frac{\de}{\de t}\omega(t).$

If we renormalize the metrics by dividing by $t$ and we let $t$ go to infinity we get
$$\frac{\omega(t)}{t}\to \frac{1}{2y^2}\mn dz\wedge d\ov{z},$$
smoothly on $S$. The limit degenerate metric is simply the pullback $\pi^*\omega_{\mathrm{KE}}$ of the Poincar\'e metric  from $H/\Gamma=C$, which is the base of the fibration. This metric on $C$ satisfies
$\Ric(\omega_{\mathrm{KE}})=-\omega_{\mathrm{KE}}$, and we will write $d_{\mathrm{KE}}$ for its distance function.

\begin{theorem}\label{elliptic} Let $\pi:S\to C$ be a minimal non-K\"ahler properly elliptic surface which is an elliptic bundle, and
let $\omega_0$ be the initial Gauduchon metric we just described.
As $t$ approaches $+\infty$ we have that $\frac{\omega(t)}{t}$ converge to $\pi^*\omega_{\mathrm{KE}}$
in $C^\infty(S,\omega_0)$, and also $\left(S, \frac{\omega(t)}{t}\right)\overset{GH}{\to} (C, d_{\mathrm{KE}})$.
\end{theorem}

Formally, the behavior is exactly the same as for the K\"ahler-Ricci flow on a K\"ahler elliptic surface $\pi:S\to C$
which is a fiber bundle over a curve $C$ of genus at least $2$ \cite{ST} (see also \cite{FZ, GTZ}).

\begin{proof}
The fact that $\frac{\omega(t)}{t}$ converge smoothly to $\pi^*\omega_{\mathrm{KE}}$ has already been proven.
We now show the Gromov-Hausdorff convergence. For any curve $\gamma\subset S$ let $L_t(\gamma)$ be its length measured in
the metric $\frac{\omega(t)}{t}$, and denote by $d_t$ the distance function of $\frac{\omega(t)}{t}$.
Let $F=\pi:S\to C$ and define a map $G:C\to S$ by sending every point $a\in C$ to
some chosen point in $S$ on the fiber $\pi^{-1}(a)$. The map $G$ is not canonical and usually discontinuous, and it satisfies
$F\circ G=\mathrm{Id}$ so
\begin{equation}\label{one4}
d_{\mathrm{KE}}(a,F(G(a)))=0,
\end{equation}
for all $a\in C$.
Since $\frac{\omega(t)}{t}$ restricted to every fiber of $\pi$ converges smoothly to zero,
it follows that for every $p\in S$ we have
\begin{equation}\label{two4}
d_t(p,G(F(p)))\leq \ve,
\end{equation}
for all $t$ large.
Given any two points $p,q\in S$, call $a=F(p), b=F(q)$ and fix a geodesic $\gamma(s)$ for the Poincar\'e metric $\omega_{\mathrm{KE}}$ that joins them.
Choose then a lift $\ti{\gamma}(s)$ to a curve in $H$ from a point $\ti{a}$ (that projects to $a$) to $\ti{b}$ (that projects to $b$).
Define a curve $\ti{\sigma}(s)=(\ti{\gamma}(s),1)$ in $H\times\mathbb{C}^*$, and call $\sigma(s)$ its projection to $S=(H\times\mathbb{C}^*)/\Gamma$,
so that $F(\sigma(s))=\gamma(s)$. The length of $\sigma(s)$ with respect to $\frac{\omega(t)}{t}$ equals the length of $\ti{\sigma}(s)$ with respect
to the pullback of $\frac{\omega(t)}{t}$ to $H\times\mathbb{C}^*$. Since the tangent vector to $\ti{\sigma}(s)$ is $(\ti{\gamma}'(s),0)$,
we see from \eqref{metric2} that
$$|\ti{\sigma}'(s)|^2_{\frac{\omega(t)}{t}}=\left(\frac{1}{2}+\frac{2}{t}\right) \frac{1}{y(s)^2},$$
where $y(s)=\mathrm{Im} \ti{\gamma}(s)$. On the other hand
$$|\ti{\gamma}'(s)|^2_{\omega_{\mathrm{KE}}}= \frac{1}{2y(s)^2},$$
and so
$$|L_t(\sigma)-d_{\mathrm{KE}}(a,b)|\leq \ve,$$
if $t$ is large. On the other hand, the $p$ and the initial point of $\sigma$ lie on the same fiber of $\pi$, so
their $d_t$-distance is less than $\ve$ for $t$ large, and similarly for $q$ and the end point of $\sigma$. Therefore
\begin{equation}\label{three4}
d_t(p,q)\leq 2\ve+L_t(\sigma)\leq d_{\mathrm{KE}}(F(p),F(q))+3\ve.
\end{equation}
Since $F\circ G=\mathrm{Id}$, we also have that
\begin{equation}\label{four4}
d_t(G(a),G(b))\leq d_{\mathrm{KE}}(a,b)+3\ve.
\end{equation}
Given now two points $p,q\in S$, let $\gamma$ be a curve joining $p$ and $q$ with $L_t(\gamma)=d_t(p,q)$.
Then arguing in a similar way to the proof of Theorem \ref{inoue},
\begin{equation}\label{five4}
d_{\mathrm{KE}}(F(p),F(q))\leq L_{\mathrm{KE}}(F(\gamma)) = L_{\pi^*\omega_{\textrm{KE}}}(\gamma) \le  L_t(\gamma)= d_t(p,q),
\end{equation}
where $L_{\pi^*\omega_{\textrm{KE}}}(\gamma)$ is the length of $\gamma$ with respect to the degenerate metric $\pi^*\omega_{\textrm{KE}}.$
This also implies that for $a, b \in S^1$,
\begin{equation}\label{six4}
d_{\mathrm{KE}}(a,b)\leq d_t(G(a),G(b)).
\end{equation}
Combining \eqref{one4}, \eqref{two4}, \eqref{three4}, \eqref{four4}, \eqref{five4} and \eqref{six4} we get the required Gromov-Hausdorff convergence.
\end{proof}

We now consider the general case, when $\pi:S\to C$ is not a fiber bundle. From \cite[Lemma 1]{Bri} or \cite[Lemma 7.2]{Wa} we see that
$\pi$ has no singular fibers, but in general it might have multiple fibers. Let us call $D\subset S$ the set of all multiple fibers of $\pi$, so that $\pi(D)$ consists of finitely many points. Then again from \cite{Ma} (see also \cite[Proposition 2]{Be}, \cite[Theorem 7.4]{Wa})
we have that $S$ is a quotient of $H\times\mathbb{C}^*$ by a discrete subgroup $\Gamma'$ of $SL(2,\mathbb{R})\times\mathbb{C}^*$, which acts by
$$\left(\left(\begin{array}{cc} a & b  \\ c & d \end{array}\right), t\right)\cdot(z,w)=\left( \frac{az+b}{cz+d}, (cz+d)\cdot w\cdot t\right),$$
and the map $\pi:S\to C$ is again induced by the projection $H\times\mathbb{C}^*\to H$.
The previous case is obtained by mapping
$SL(2,\mathbb{R})\times\mathbb{Z}\to SL(2,\mathbb{R})\times\mathbb{C}^*$ by
$(A,n)\mapsto (A,\alpha^n \chi(A))$.

If we consider the projection $\Gamma$ of $\Gamma'$ to $SL(2,\mathbb{R})$,
we now have that in general the $\Gamma$-action on $H$ is not free, so its quotient $C=H/\Gamma$ is an orbifold
(it is actually a ``good'' orbifold, i.e. a global finite quotient of a manifold, see \cite[p. 139]{Wa} where it is shown that if $C$ was a ``bad'' orbifold
then we would have that $\kappa(S)=-\infty$).
The finitely many orbifold points of $C$ are precisely equal to $\pi(D)$.

The $(1,0)$-forms on $H\times\mathbb{C}^*$
$$-\frac{2}{w}dw +\frac{\mn}{y}dz, \quad \frac{1}{y^2}dz\wedge d\ov{z}$$
are still invariant under the $\Gamma'$-action, and so
again they descend to $S$. We then define $\omega_0$ as before, and also $\omega_{\mathrm{KE}}$ which is now
an orbifold K\"ahler-Einstein metric on $C$. It is easy to see that it induces a distance function $d_{\mathrm{KE}}$ on $C$ (see e.g. \cite{Bo}). On the other hand the $(1,1)$ form $\pi^*\omega_{\mathrm{KE}}$ is
smooth on all of $S$.

\begin{theorem}\label{elliptic2} Let $\pi:S\to C$ be a general minimal non-K\"ahler properly elliptic surface with initial Gauduchon metric $\omega_0$ described above.
As $t$ approaches $+\infty$ we have that $\frac{\omega(t)}{t}$ converge to $\pi^*\omega_{\mathrm{KE}}$
in $C^\infty(S,\omega_0)$, and also $\left(S, \frac{\omega(t)}{t}\right)\overset{GH}{\to} (C, d_{\mathrm{KE}})$.
\end{theorem}
\begin{proof}
The fact that $\frac{\omega(t)}{t}$ converge smoothly to $\pi^*\omega_{\mathrm{KE}}$ follows from the same calculation as in the previous case. The Gromov-Hausdorff convergence can be proved as follows. We know that there is another properly elliptic surface $\pi':S'\to C'$ which is an elliptic bundle over $C'$ with $g(C')\geq 2$, with a finite group $\Gamma''$ acting on $S'$ and $C'$ (so that $\pi'$ is $\Gamma''$-equivariant) such that the $\Gamma''$-action on $S'$ is free, while the $\Gamma''$-action on $C'$ is not, and
$\pi:S\to C$ is equal to the $\Gamma''$-quotient of $\pi':S'\to C'$. If we call $p:S'\to S$ and $q:C'\to C$ the quotient maps then we have that $p^*\omega_0$ equals the same metric $\omega_0$ from the earlier discussion, and similarly $p^*\omega(t)$ equals the evolved metrics on $S'$, and $\Gamma''$ acts by isometries of $p^*\omega(t)$. Also, the pullback distance $q^*d_{\mathrm{KE}}$ on $C'$ equals
the K\"ahler-Einstein distance function from earlier, and again $\Gamma''$ acts by isometries of $q^*d_{\mathrm{KE}}$. From Theorem \ref{elliptic} we know that
$\left(S', \frac{p^*\omega(t)}{t}\right)$ converges in Gromov-Hausdorff to $(C',q^*d_{\mathrm{KE}})$.

In fact, we claim that this convergence happens also in the $\Gamma''$-equivariant Gromov-Hausdorff sense.
Indeed, let us consider the same maps $F:S'\to C'$ and $G:C'\to S'$ from the proof of Theorem \ref{elliptic}.
Then $F$ is $\Gamma''$-equivariant, while in general $G$ is not, but for any element $g\in \Gamma''$ and for
any point $x\in C'$ the points $g\cdot G(x)$ and $G(g\cdot x)$ lie in the same fiber of $F$,
and so their distance with respect to the metric $\frac{p^*\omega(t)}{t}$ goes to zero as $t$ approaches zero (uniformly in $x$ and $g$).
This shows that the $\Gamma''$-equivariant Gromov-Hausdorff distance between $\left(S', \frac{p^*\omega(t)}{t}\right)$ and $(C',q^*d_{\mathrm{KE}})$ is less than
any $\ve>0$ if $t$ is small enough.
Then \cite[Theorem 2.1]{Fuk} or \cite[Lemma 1.5.4]{Ro} imply that $\left(S, \frac{\omega(t)}{t}\right)\overset{GH}{\to} (C, d_{\mathrm{KE}})$.
\end{proof}

Theorems \ref{elliptic} and \ref{elliptic2} together complete the proof of part (c) of Theorem \ref{surfaces}.

\section{The Mabuchi energy} \label{sectionmab}

In this section, we show that the Mabuchi energy functional from K\"ahler geometry can be defined in the setting of a complex surface with vanishing first Bott-Chern class, and that it is decreasing along the Chern-Ricci flow.

Let $M$ be a surface with vanishing first Bott-Chern class, and let $\omega_0$ be a Gauduchon metric on $M$.  By definition of the Bott-Chern class, there exists a unique function $F$ with
$$\textrm{Ric}(\omega_0) = \ddbar F, \quad \int_M e^F \omega_0^n=0.$$
Define  $\mathcal{H}$  to be the space of all Gauduchon metrics $\omega'$ on $M$ of the form $\omega' =\omega_0 + \ddbar \psi$ for some smooth function $\psi$.  We then define the \emph{Mabuchi energy} $\textrm{Mab}_{\omega_0}: \mathcal{H} \rightarrow \mathbb{R}$ by
$$\textrm{Mab}_{\omega_0} (\omega') = \int_M \left( \log \frac{\omega'^2}{\omega_0^2} - F\right)\omega'^2 + \int_M F \omega_0^2.$$
Comparing with the formula given in \cite{T}, one can check that this coincides with the Mabuchi energy in the K\"ahler setting.

Now let $\omega(t)$ solve the Chern-Ricci flow starting at $\omega_0$.  Then we may write $\omega(t) = \omega_{\varphi}:= \omega_0+ \ddbar \varphi$ where $\varphi=\varphi(t)$ solves
$$\ddt{\varphi} = \log \frac{\omega_{\varphi}^2}{\omega_0^2} - F, \quad \varphi|_{t=0} =0.$$

The result of this section is that the Mabuchi energy decreases along the Chern-Ricci flow.

\begin{proposition} \label{mab} With the notation as above, we have
$$\ddt{} \emph{Mab}_{\omega_0} (\omega_{\varphi}) = - 2 \int_M \sqrt{-1} \partial \dot{\varphi} \wedge \overline{\partial} \dot{\varphi} \wedge \omega_{\varphi} \le 0.$$
\end{proposition}
\begin{proof}
Compute
\begin{align*}
\ddt{} \textrm{Mab}_{\omega_0} (\omega_{\varphi}) & = \int_M \Delta_{\omega_{\varphi}} \dot{\varphi} \, \omega_{\varphi}^2 + \int_M \left( \log \frac{\omega_{\varphi}^2}{\omega_0^2} - F\right) \Delta_{\omega_{\varphi}} \dot{\varphi} \, \omega_{\varphi}^2 \\
& = \int_M \dot{\varphi} \Delta_{\omega_{\varphi}} \dot{\varphi} \, \omega_{\varphi}^2 = 2 \int_M \dot{\varphi} \sqrt{-1} \partial \ov{\partial} \dot{\varphi} \wedge \omega_{\varphi}.
\end{align*}
Integrating by parts,
\begin{align*}
\ddt{} \textrm{Mab}_{\omega_0} (\omega_{\varphi})
& = - 2 \int_M \sqrt{-1} \partial \dot{\varphi} \wedge \overline{\partial} \dot{\varphi} \wedge \omega_{\varphi} + 2 \int_M \dot{\varphi} \sqrt{-1} \, \ov{\partial} \dot{\varphi} \wedge \partial \omega_0 \\
& = - 2 \int_M \sqrt{-1} \partial \dot{\varphi} \wedge \overline{\partial} \dot{\varphi} \wedge \omega_{\varphi} +  \int_M \sqrt{-1}\, \ov{\partial} \dot{\varphi}^2 \wedge \partial  \omega_0 \\
& =  - 2 \int_M \sqrt{-1} \partial \dot{\varphi} \wedge \overline{\partial} \dot{\varphi} \wedge \omega_{\varphi},
\end{align*}
using the fact that $\ddbar \omega_0=0$.
\end{proof}

Recall that Gill \cite{G} showed, in the setting of vanishing first Bott-Chern class, the Chern-Ricci flow $\omega(t)$ starting at any Hermitian metric $\omega_0$ converges in $C^{\infty}$ to a Chern-Ricci flat metric $\omega_{\infty}$.
Proposition \ref{mab} can be used to give an alternative proof of the convergence part of  Gill's theorem  in the special case when $M$ has complex dimension two and $\omega_0$ is Gauduchon.  Indeed the proof follows exactly as in the K\"ahler case.  It was first noted in unpublished work of H.-D. Cao that the Mabuchi energy decreases along the K\"ahler-Ricci flow, and to see how Proposition \ref{mab} implies convergence,  one can follow the arguments of Phong-Sturm \cite{PS}.   Alternatively, see the exposition in section 4 of \cite{SW3}.\\

\noindent {\bf Acknowledgements:} Part of this work was carried out
while the authors were visiting the Centre de Recherches Math\'ematiques in
Montr\'eal and while the first-named author was visiting the Mathematical Science
Center of Tsinghua University in Beijing. We would like to thank these institutions
for their hospitality, and also V. Apostolov for helpful
discussions.

\end{document}